\documentclass[11pt,reqno]{amsart}

\usepackage{amssymb}
\usepackage{amsmath}
\usepackage{hyperref}       
\usepackage{url}            
\usepackage{booktabs}       
\usepackage{amsfonts}       
\usepackage{microtype}      
\usepackage{graphicx}
\usepackage{mathtools}
\usepackage{enumitem}
\usepackage{comment}
\usepackage{parskip}
\usepackage[capitalize]{cleveref}
\usepackage{tikz-cd}
\usepackage{subcaption}
\usepackage{listings}
\usepackage{multirow}
\usepackage[hyperref,doi,url=false,style=alphabetic,maxbibnames=99]{biblatex}

\AtEveryBibitem{
	 \clearlist{address}
	  \clearfield{date}
	    \clearfield{isbn}
		 \clearfield{issn}
		  \clearlist{location}
		   \clearfield{month}
		    \clearfield{series}
		    \clearfield{note}
			 
			 \ifentrytype{book}{}{
			   \clearlist{publisher}
			     \clearname{editor}
			  }
}
\theoremstyle{plain}
\newtheorem{thm}{Theorem}[section]
\newtheorem{lem}[thm]{Lemma}
\newtheorem{prop}[thm]{Proposition}
\newtheorem{cor}[thm]{Corollary}
\newtheorem{conj}[thm]{Conjecture}
\newtheorem{ass}[thm]{Assumption}
\newtheorem{rmk}[thm]{Remark}
\newtheorem{conv}[thm]{Convention}

\theoremstyle{remark}

\theoremstyle{definition}
\newtheorem{defn}[thm]{Definition}
\newtheorem{ex}[thm]{Example}
\newtheorem{cnstr}[thm]{Construction}

\newcommand{\Z}{\mathbb{Z}}
\newcommand{\Q}{\mathbb{Q}}
\newcommand{\N}{\mathbb{N}}
\newcommand{\R}{\mathbb{R}}
\newcommand{\set}[1]{\left\{\,#1\,\right\}}
\newcommand{\paren}[1]{\left(#1\right)}

\newcommand{\inv}{^{-1}}
\newcommand{\M}{M_\phi}
\newcommand{\Mpunc}{M_\phi^\circ}
\newcommand{\Sigmapunc}{\Sigma^\circ}
\newcommand{\MM}{\M(\textbf{p};\textbf{q})}
\newcommand{\corei}{K_i(\textbf{p};\textbf{q})}

\newcommand{\F}{\mathcal F}
\newcommand{\hatF}{\widetilde{\F}}
\renewcommand{\H}{\mathcal H}
\newcommand{\guts}{\mathcal G}
\newcommand{\interst}{\mathcal I}
\newcommand{\bd}{\partial}
\newcommand{\fequiv}{\sim}
\newcommand{\J}{\mathcal J}
\newcommand{\Jpar}{\mathcal J^{\text{partial}}}
\newcommand{\tmax}{t_{\max{}}}
\newcommand{\eps}{\varepsilon}
\renewcommand{\phi}{\varphi}
\newcommand{\titletex}[1]{\texorpdfstring{#1}{\detokenize{#1}}}

\newcommand{\plabel}[1]{}

\usepackage{xstring}
\usepackage{xcolor}
\newcommand{\xlabel}[2]{\IfSubStr{#1}{fix8}{\phantomsection\label{#1} \textcolor{red}{#2}}{#2}}
\newcommand{\mfldcount}{2598}
\newcommand{\mfldpercent}{44.7}

\DeclareMathOperator{\xright}{east}
\DeclareMathOperator{\xleft}{west}
\DeclareMathOperator{\xess}{ess}
\DeclareMathOperator{\xiness}{iness}
\DeclareMathOperator{\xtan}{tan}

\DeclareMathOperator{\Stab}{Stab}
\DeclareMathOperator{\Area}{Area}
\DeclareMathOperator{\Homeo}{Homeo}
\DeclareMathOperator{\rank}{rank}
\DeclareMathOperator{\PSL}{PSL}
\DeclareMathOperator{\interior}{int}

\graphicspath{{figs/}}

\begin{document}
\pagenumbering{gobble}
\pagenumbering{arabic}

\title{Taut foliations, left-orders, and pseudo-Anosov mapping tori}
\author{Jonathan Zung}
\address{Department of Mathematics, Princeton}
\email{jzung@math.princeton.edu}
\begin{abstract}
	For a large class of 3-manifolds with taut foliations, we construct an action of $\pi_1(M)$ on $\R$ by orientation preserving homeomorphisms which captures the transverse geometry of the leaves. This action is complementary to Thurston's universal circle. Applications include the left-orderability of the fundamental groups of every non-trivial surgery on the figure eight knot. Our techniques also apply to at least \mfldcount{} manifolds representing \mfldpercent{}\% of the non-L-space rational homology spheres in the Hodgson-Weeks census of small closed hyperbolic 3-manifolds.
\end{abstract}

\maketitle
\begin{section}{Introduction}\label{sec:intro}
	A taut foliation on a 3-manifold $M$ is a valuable structure. Taut foliations may be used to certify nontriviality of a transverse loops or to certify surfaces of minimum genus in their homology classes, as shown in classic work of Novikov, Roussarie, Thurston, and Gabai \cite{novikov_topology_1965, roussarie_plongements_1974, thurston_norm_1986, gabai_foliations_1983}. The existence of a taut foliation puts constraints on $\pi_1(M)$. For example, Thurston showed that if $M$ is an atoroidal, irreducible 3-manifold with a taut foliation, then $\pi_1(M)$ admits a faithful action on $S^1$ by homeomorphisms~\cite{thurston_three-manifolds_1997, calegari_laminations_2003}\plabel{fix2:1}. In another direction, Kronheimer and Mrowka made a connection to Floer theory via contact and symplectic geometry, showing that a taut foliation gives rise to a nontrivial class in monopole Floer homology \cite{kronheimer_scalar_1997, kronheimer_monopoles_2007}. Ozsv\'{a}th and Szab\'{o} established parallel results in the setting of Heegaard Floer homology \cite{ozsvath_holomorphic_2004}.

	The L-space conjecture, formulated in parts by Boyer-Gordon-Watson and Ozsv\'{a}th-Szab\'{o}, is a proposed sharpening of the connections outlined above \cite{boyer_l-spaces_2011, ozsvath_holomorphic_2004}. It posits that the following are equivalent for an orientable, irreducible rational homology sphere $M$:

	\begin{enumerate}
		\item $M$ has a co-orientable taut foliation
		\item $\pi_1(M)$ is left-orderable (ie $\pi_1(M)$ acts faithfully on $\R$ by orientation preserving homeomorphisms).
		\item $M$ is not an L-space (ie its Heegaard Floer homology $\widehat{HF}(M)$ satisfies the strict inequality $\rank(\widehat{HF}(M))>|H^2(M;\Z)|$).
	\end{enumerate}

	Technology for deciding conditions (1) and (3) is well developed; for example, Dunfield verified the equivalence of (1) and (3) for 99.8\% of the manifolds in his census of ${\sim}300,000$ small hyperbolic rational homology spheres \cite{dunfield_floer_2019}.  Techniques for deciding (2) are harder to come by. For every non-left-orderable group, there is a finite length certificate proving its non-left-orderability. On the other hand, left-orderability is not known to be decidable for three-manifold groups. See Calegari and Dunfield \cite{calegari_laminations_2003} for a discussion\plabel{fix2:2}. Here are a few practical methods for proving the left-orderability of a 3-manifold group:

	\begin{itemize}
		\item \label{item:rep} Try to lift $\PSL(2,\R)$ representations of $\pi_1(M)$ to $\widetilde{\PSL(2,\R)}$ which acts on the the universal cover of the circle at infinity in $\mathbb{H}^2$. See Eisenbud et al or Culler-Dunfield for examples of this technique\cite{eisenbud_transverse_1981,culler_orderability_2018}.
		\item \label{item:circle} If $M$ has a taut foliation, try to lift Thurston's action of $\pi_1(M)$ on $S^1$ to an action on $\R$. This works whenever the Euler class of the plane field tangent to the foliation vanishes. See Calegari-Dunfield or Boyer-Clay for examples of this technique \cite[Section~7]{calegari_laminations_2003} \cite[Section~5]{boyer_taut_2019}.
		\item We say that a foliation $\F$ is \emph{$\R$-covered} if the leaf space of the lift of $\F$ to the universal cover of $M$ is homeomorphic to $\R$ (eg in the case of the foliation of a fibered 3-manifold by fiber surfaces). Since $\pi_1(M)$ always acts on the leaf space, we get an action on $\R$.
	\end{itemize}

	The third technique is appealing since it directly uses the transverse geometry of the foliation, but is limited in generality since most taut foliations are not $\R$-covered. In this paper, we demonstrate a method for improving the third technique to work for more general taut foliations. We study a family of 3-manifolds with taut foliations which are not $\R$-covered, but whose leaf spaces admit a map to $\R$ such that the action of $\pi_1$ descends to an action on $\R$. The question of the existence of taut foliations and left-orderings compatible in this sense was first raised by Thurston \cite[Section~8.1]{calegari_problems_2002}\cite{thurston_three-manifolds_1997}.

	We prove the following:
	\begin{thm}\label{thm:main}
		Let $\Sigma$ be an orientable closed surface and $\phi\colon\Sigma\to\Sigma$ a pseudo-Anosov map with orientable invariant foliations. Suppose further that $\phi$ preserves the orientation of these foliations. Let $\M$ be the mapping torus of $\phi$. Let $\MM$ be the result of non-zero surgery along any collection of closed orbits of $\phi$. If the surgery slopes all have the same sign, then $\MM$ has left-orderable fundamental group.
	\end{thm}

	Here we take the zero slope, also known as the \emph{degeneracy slope}, to be the one which crosses no prongs. See \cref{conv:slopes} for a full explanation of our slope conventions.

	\begin{thm}\label{thm:main2}
		With the assumptions of \cref{thm:main}, there is a taut foliation $\F$ on $\MM$. Let $L$ be the leaf space of $\hatF$. Then there is a continuous, monotone map $f\colon L \to \R$ so that the action of $\pi_1(\MM)$ on $L$ descends to a nontrivial action on $\R$.
	\end{thm}

	Here monotone means that $f$ respects the natural partial order on $L$, ie if leaf $\lambda_1$ may be connected to leaf $\lambda_2$ by an oriented arc positively transverse to $\hatF$, then $f(\lambda_1)\leq f(\lambda_2)$. \label{fix2:3}

	\begin{cor}\label{cor:fig8}
		For any $n\geq 1$, any non-trivial surgery on the $n$-fold cyclic branched cover of the figure-eight knot has left orderable fundamental group.
	\end{cor}

	Previously, orderability for surgeries on the figure eight knot was known only for slopes in $[-4,4]\cup \Z$. The range $(-4,4)$ was treated by the representation theoretic approach of Boyer, Gordon, and Watson \cite{boyer_l-spaces_2011}, while the toroidal exceptional surgeries $\set{4,-4}$ were resolved by Clay, Lidman, and Watson \cite{clay_graph_2013} using a gluing theorem for amalgamations of left-orderable groups. Fenley's work on $\R$-covered Anosov flows applies to show that integer surgeries on the figure eight knot are left-orderable \cite{fenley_structure_1998}. Hu recently gave another approach to the case of integer slopes by showing that certain taut foliations on these manifolds have vanishing Euler class \cite{hu_euler_2019}. Hu also proves a negative result: for slopes outside $[-2,2] \cup \Z$, there does not exist a co-orientable taut foliation which both has trivial Euler class and is transverse to the Dehn surgery core. Thus, the action of the universal circle on such foliations does not lift to an action on $\R$. We discuss this further in \cref{rmk:eulerzero}.

	In \cref{sec:foliations}, we set up notation and construct taut foliations on the manifolds of \cref{thm:main}. In \cref{sec:branching}, we analyze the branching in the leaf spaces of these foliations.

	Our approach to defining the map $f\colon L\to \R$ is to glue together certain branches of the leaf space $L$. This point of view is outlined in \cref{sec:gluing}. More formally, in \cref{sec:connection} we define an $\R$-bundle with structure group $\Homeo^+(\R)$ and a flat \emph{partial connection} $\Jpar$. We complete this bundle by adding a point at infinity to each fiber. The resulting $S^1$ bundle has an honest flat connection $\J$. Moreover, this $S^1$ bundle has vanishing Euler class and hence lifts to a flat $\R$ bundle.

	Finally, in \cref{sec:computations} we report on computations showing that manifolds satisfying the hypotheses of \cref{thm:main} are abundant in the Hodgson-Weeks census.

	\begin{subsection}*{Acknowledgements}
		The author would like to thank David Gabai and Sergio Fenley for several helpful discussions about this work. Nathan Dunfield and Mark Bell generously shared data from their census of monodromies of small hyperbolic manifolds. A special thanks is owed to the referee whose exceptionally detailed and patient reading resulted in many corrections and improvements to the article. The author is indebted to Peter Ozsv\'{a}th for his consistent encouragement and guidance during this project.
	\end{subsection}
\end{section}

	\begin{section}{Notation and conventions}
	In this section, we set up notation that will be used throughout the article. Let $\Sigma$ be a closed orientable surface. Let $\phi\colon\Sigma\to\Sigma$ be a pseudo-Anosov map such that its stable and unstable foliations, denoted $\F^s|_\Sigma$ and $\F^u|_\Sigma$, are orientable. This implies that each singularity of $\phi$ has an even number of prongs. Suppose further that $\phi$ preserves the orientations of $\F^s|_\Sigma$ and $\F^u|_\Sigma$. 

	Let $M_\phi$ denote the mapping torus of $\phi$, and let $K_1,\dots,K_n \subset \M$ be the suspensions of any $n$ periodic orbits of $\phi$. For ease of exposition, we will always assume that the suspensions of the singularities of $\phi$ are included in $K_1,\dots,K_n$. 	
	
	Let $\Mpunc=\M\setminus\set{K_i}$. We will use $\MM$ to denote the manifold obtained by slope $(p_i;q_i)$ surgery along $K_i$. Let $\corei$ denote the core of the Dehn filling of $K_i$ in $\MM$. Our slope conventions are explained below in \cref{conv:slopes}.\plabel{fix3:5}.

	Let $\Lambda|_\Sigma$ be the stable invariant lamination of $\phi$ produced by splitting open $\F^s|_\Sigma$ along the prongs at each singularity. Let $\F^s$, $\F^{u}$, and $\Lambda$ be the suspensions of $\F^s|_\Sigma$, $\F^u|_\Sigma$, and $\Lambda|_\Sigma$ in $M_\phi$. (To be precise, we should take the suspension of $\Lambda|_\Sigma$ using an appropriate blow-up of the flow $\phi$.)\plabel{fix3:2} The orientability constraints are equivalent to the orientability of $\F^s$. \plabel{fix:33}

	\begin{ex} The figure eight knot complement fibers over the circle with a genus 1 fiber and pseudo-Anosov monodromy. We can choose coordinates on the fiber $T^2 \setminus (0,0)$ so that the monodromy is $\begin{psmallmatrix} 2 & 1 \\ 1 & 1 \end{psmallmatrix}$. Since this matrix has distinct positive real eigenvalues, the monodromy preserves the orientation of the invariant foliations as desired.
	\end{ex}

	\begin{proof}[Proof of \cref{cor:fig8}]
		Consider the $n^{th}$ cyclic branched cover of the figure eight knot. The lift of the figure eight knot to this cyclic branched cover is a fibered knot with monodromy $\begin{psmallmatrix} 2 & 1 \\ 1 & 1 \end{psmallmatrix}^n$. With our slope conventions, surgery along the zero slope (ie the degeneracy slope) yields the $n^{th}$ cyclic branched cover. By \cref{thm:main}, surgery along any non-zero slope yields a manifold with left-orderable fundamental group.\plabel{fix3:3}
	\end{proof}

	\begin{ex} We can generate examples with a given fiber genus and prescribed singularities by enumerating periodic splitting sequences of train tracks, as in \cite{penner_combinatorics_2016}. One of the lowest volume examples appearing in the genus 2, 1-singularity enumeration is the 1-cusped hyperbolic manifold m038.
	\end{ex}

	\begin{figure}[ht!]
		\centering
		\includegraphics[width=70mm]{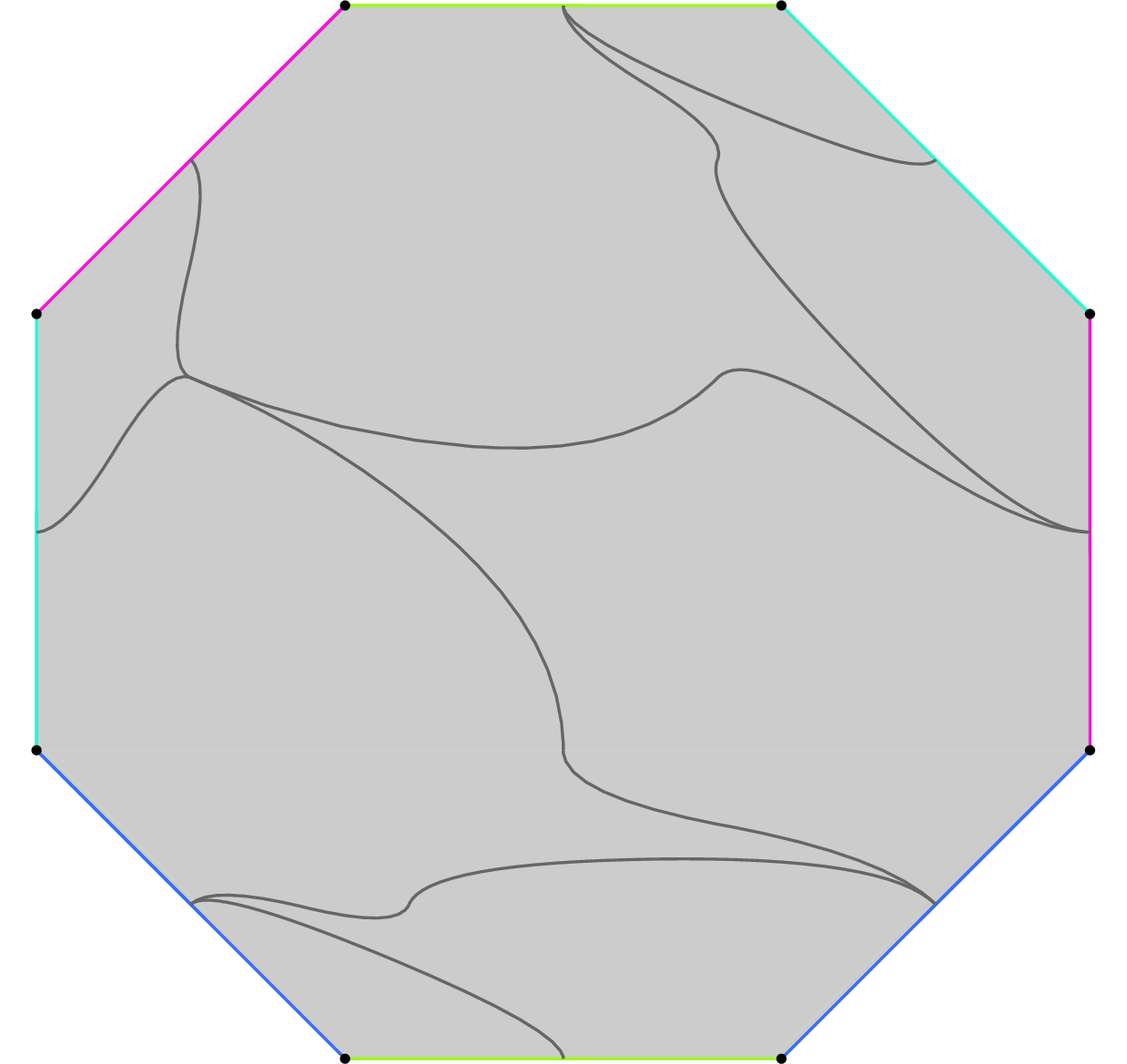}
		\caption{This figure shows a train track carrying an invariant lamination for the monodromy of the fibered, 1-cusped 3-manifold m038. The fiber surface has genus 2 with one puncture at the vertices of the octagon. The complement of the train track is a punctured ideal hexagon. The figure was generated using Mark Bell's program \texttt{flipper} \cite{bell_flipper_2013}.}
	\end{figure}

	\begin{ex} The $(-2,3,7)$ pretzel knot is fibered with a genus 5 fiber and pseudo-Anosov monodromy. In this case, the monodromy $\phi$ has a single 18-prong singularity at the boundary of the fiber, so the invariant foliations are orientable. However, $\phi$ reverses this orientation. The branched double cover of the $(-2,3,7)$ pretzel knot does preserve the orientation of the invariant foliations, and so satisfies the given condition.
	\end{ex}

	\begin{conv}[Slope conventions]\label{conv:slopes}
		For each $i$, let $m_i$ be the period of the orbit $K_i$. The $m_i$ singularities in $K_i\cap\Sigma$ have the same number, denoted $k_i$, of prongs. We assume throughout that $k_i$ is even. Define $\omega_i$ to be the integer in $[0,k_i)$ such that $2\pi \omega_i$ is the counterclockwise angle by which $\phi^{m_i}$ rotates one of these singularities. We are thinking of each singularity is a cone point with angle $2\pi k_i$. \plabel{fix3:4} \plabel{fix2:5}

		We describe slopes in a slightly nonstandard way. A slope of $(p_i;q_i)$ corresponds with a curve in $\bd N(K_i)$ defined as follows. Choose a point near $K_i\cap \Sigma$ and flow it along the suspension flow of $\phi$ for time $m_iq_i$\plabel{fix2:4.7}. (Here we set the speed of the flow that the time 1 flow of the suspension flow intersects each fiber once and returns to the starting fiber.) Typically, the resulting curve will not close up since $\phi^{m_i}$ may rotate singularities. We close it by appending a path in $\Sigma$ which walks around the relevant singularity by a clockwise angle of $2\pi p_i$. Not all such pairs represent slopes; a pair $(p_i;q_i)$ corresponds with a closed curve if and only $p_i = \omega_iq_i \mod k_i$.

		We always assume $q_i\geq 0$. We say that a slope is positive, resp. negative when $p_i$ is positive, resp. negative. We say that the slope is $\infty$ when $q_i=0$. The $\infty$ slope corresponds with the fiber slope and is declared to be both positive and negative. The zero slope or the degeneracy slope is $(p_i;q_i)=(0;k_i/\gcd(\omega_i,k_i))$. In an abuse of notation, the zero slope and the infinity slope may intersect more than once.
	\end{conv}

	$\F$ will denote the taut foliation to be constructed in \cref{cnstr:foln}. Let $\guts$ denote the guts of $\Lambda$, ie the compact subspace of $M\setminus \Lambda$ obtained by chopping off the ends of the ideal polygon bundles. Let $\interst$ denote the interstitial region, ie the part of $M\setminus \Lambda$ we just cut off. Topologically, $\interst$ is a disjoint union of I-bundles over half-infinite cylinders. We have $\MM = \Lambda \cup \guts \cup \interst$. Note that the distinct components of $\guts$ have disjoint closures.

	It will also be convenient to work with another decomposition $\MM = \Lambda' \cup \guts'$. We obtain $\Lambda'$ by blowing down $\interst$ and define $\guts'$ to be the closure of $\MM\setminus \Lambda'$. Some of the leaves of $\Lambda'$ are branched surfaces instead of surfaces, so $\Lambda'$ is not a lamination, but a branched lamination. The advantage of $\Lambda'$ over $\Lambda$ is that there is a Solv metric supported on $\Lambda'$.

	\begin{figure}[ht]
		\def\svgwidth{0.6\linewidth}
		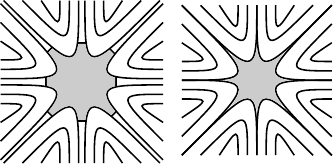
		\caption{Blowing down $\interst$. The shaded region on the left is $\guts$, and the shaded region on the right is $\guts'$.}
	\end{figure}

	The following lemma will be useful later:
	\begin{lem}\label{lem:lambdaprime}
		\renewcommand{\Sigmapunc}{S}
		Let $\Sigmapunc$\plabel{fix2:19} be a fiber surface in the fibered manifold $\MM\setminus \guts$. Suppose $\gamma$ is a path in $\Sigmapunc \cap (\MM \setminus \guts)$ with endpoints on two walls of $\guts$. Suppose that $\gamma$ makes $N_{\tan}$ tangencies with $\F$. Then there is an arc $\gamma'$ in $\Sigmapunc \cap (\MM\setminus \guts')$ with endpoints on the corresponding two walls of $\guts'$ and in the same homotopy class rel. endpoints as $\gamma$ such that $\gamma'$ makes at most $N_{\tan}$ tangencies with $\Lambda'$. (Here the homotopy happens in $\Sigmapunc \cap (\MM \setminus \guts)$ \plabel{fix3:6}and we are identifying relative homotopy classes of paths in $(\Sigmapunc \cap (\MM\setminus \guts),\Sigmapunc \cap \partial \guts)$ and $(\Sigmapunc \cap (\MM \setminus \partial \guts'), \Sigmapunc \cap \partial \guts')$ in the obvious way.) The converse also holds.
	\end{lem}
	\begin{proof}
		First, homotope $\gamma$ so that all of its tangencies with $\F$ and all of its self-intersections occur on the interior of $\interst$. This may be done without introducing any new tangencies. Let $\interst_\eps$ be the subset of $\interst$ of thickness less than $\eps$. We begin by blowing down $\interst_\eps$ for some $\eps>0$. Let $\F_\eps$ and $\Lambda_\eps$ be the images of $\F$ and $\Lambda$ under this blowdown.\plabel{fix2:7}
		
		If we choose $\eps$ small enough, then $\gamma$ intersects $\F\cap \interst_\eps$ transversely. Moreover, if we choose $\eps$ small enough, then the blowdown changes the tangent plane field of $\F$ by a uniformly small amount. Thus, we can arrange that that $\gamma$ has $N_{\tan}$ tangencies with $\F_\eps$. 

		Now, homotope $\gamma$ so that all of its tangencies with $\F_\eps$ and its self intersections occur inside $\Lambda_\eps$. At this stage, $\gamma \setminus \Lambda_\eps$ is a collection $\mathcal C$ of disjoint segments transverse to $\F_\eps$. We now blow down the rest of $\interst$ so that each of the segments in $\mathcal C$ blows down to a point. The curve $\gamma$ then blows down to a curve $\gamma'$ with the desired properties.

		The converse is easier; $\Lambda$ is obtained by splitting open the prongs in $\Lambda'$, and this may be done without introducing new tangencies with $\gamma'$.
	\end{proof}
	\end{section}

\begin{section}{Foliations on surgeries on pseudo-Anosov mapping tori}\label{sec:foliations}
		In this section, we describe the construction of taut foliations on our class of 3-manifolds. One might colloquially describe the construction as ``stuffing the guts of the suspension of the $\phi$-invariant lamination with monkey saddles".

	\begin{cnstr}\label{cnstr:suturedtorus}
		Given a sutured solid torus $D$ such that the sutures are parallel with non-meridional slope, there is a foliation of $D$ by planes compatible with the sutures. Recall that the sutures on a sutured manifold divide its boundary into (possibly disconnected) positive and negative subsurfaces.\plabel{fix:5} We can construct the desired foliation beginning with the obvious product foliation of the solid torus by disks, and then combing the edges of the disks to expose their positive sides in the positive regions of $\bd D$ and their negative sides in the negative regions of $\bd D$. This is called a foliation by a stack of monkey saddles. See \cref{fig:suturedsolidtorus}.
		\begin{figure}[ht!]
			\centering
			\def\svgwidth{90mm}
			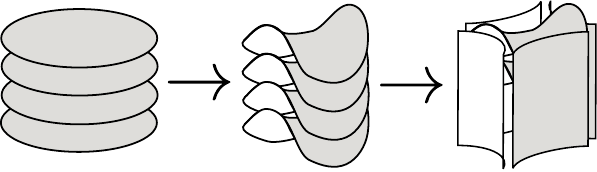
			\caption{On the left is the standard foliation of $D^2\times S^1$, where $S^1$ is cut open. In the middle, we alternately comb the edges of the disks to expose their positive and negative sides. Grey shows the positive sides of the leaves, while white shows the negative sides. On the right, we show the limiting configuration which has 4 annular leaves at the boundary which we call \emph{walls} and infinitely tall saddle-like leaves (homeomorphic to planes) on the interior. The interior saddle-like leaves accumulate on the walls. We can vary the number of legs of the saddle's rider or the gluing of the top and bottom of the picture to get sutures of any desired non-meridional slope.}
			\label{fig:suturedsolidtorus}
		\end{figure}
	\end{cnstr}

	\begin{cnstr}\label{cnstr:foln}
		Let $(p_i;q_i)$ be any choice of slopes with $p_i \neq 0$. Then the manifold $\MM$ carries a taut foliation constructed as follows. Recall that $\Mpunc=\M\setminus\set{K_i}$. Let $\F^s$ be the codimension 1 weak stable foliation of the suspension flow of $\phi$ on $\Mpunc$. Split open $\F^s$ along its prongs to obtain a lamination $\Lambda$ on $\Mpunc \subset \MM$. The complement $\MM \setminus \Lambda$ is a collection of ideal polygon bundles over $S^1$. With our conventions (see \cref{conv:slopes}), these polygons are $2p_i$-gons since a closed curve of slope $(p_i;q_i)$ decomposes into an arc of slope $(0;q_i)$, which intersects no prongs, and an arc of slope $(p_i;0)$, which wraps around an angle of $2\pi p_i$ and therefore intersects $2p_i$ prongs. Fill in these bundles with the foliations constructed in \cref{cnstr:suturedtorus}. The resulting foliation is taut because it contains no compact leaves; in fact, all leaves are either cylinders or planes.
	\end{cnstr}

	\begin{ex}
		When $p_i=1$ for all $i$, the complementary regions can be blown down without without inserting any leaves at all. The resulting 3-manifold carries an Anosov flow. Fried shows that every transitive Anosov flow with orientable invariant foliations is obtained by this construction \cite{fried_transitive_1983}. The transitivity condition is not severe; every Anosov flow on a hyperbolic 3-manifold is transitive. His construction might require the slopes to have different signs, so our results will not hold for all of these manifolds.
	\end{ex}

	\begin{rmk}\label{rmk:eulerzero}
		The Euler class of a foliation constructed via \cref{cnstr:foln} is typically nonzero, so the action of $\pi_1$ on Thurston's universal circle will not in general lift to an action on $\R$. One can see this by applying a result of Hu:

		Let $X$ be a $\Q$-homology solid torus, $k\geq 1$ the order of the longitude of $X$ in $H_2(X,\partial X)$. Let $\mu$ be any meridian. Let $S$ be the set of filling slopes for $X$ yielding a 3-manifold with a co-orientable taut foliation which is transverse to the Dehn surgery core and has zero Euler class. Let $x$ be the Thurston norm of a generator of $H_2(X,\partial X)$. If $x\neq 0$, then Theorem 5.2 of \cite{hu_euler_2019} says that
		\begin{enumerate}
			\item Outside the interval $(-x/k-1,x/k+1)$, $S$ may contain only $\mu$ and the integer slopes.
			\item $S$ is nowhere dense in $\R \cup \{1/0\}$. In particular, it is nowhere dense in $(-x/k-1,x/k+1)$.
		\end{enumerate}

		 To apply Hu's theorem, fill all but one of the boundary components of $\Mpunc$. Since $\Mpunc$ is a pseudo-Anosov suspension, it is a cusped hyperbolic manifold. Therefore, for a generic choice of slopes, the result of filling is a 1-cusped hyperbolic manifold which we call $X$. Suppose further that $X$ is a $\Q$-homology solid torus; otherwise, any filling of $X$ has $b_1>0$ and we can find a left-ordering by other means. Let $x$ be the Thurston norm of a generator of $H_2(X,\partial X)$. By hyperbolicity, $x\neq 0$. Therefore, $X$ satisfies the hypothesis of Hu's theorem. The foliations of \cref{cnstr:foln} are transverse to the Dehn surgery core because the core is transverse to the monkey saddles. Thus, in a generic choice of filling slopes for $X$, the foliations of \cref{cnstr:foln} have non-zero Euler class.

		As an example, take the figure eight knot complement. In this case, $x=1$. By Hu's result, our foliations have non-zero Euler class for filling slopes outside $(-2,2)\cup \Z$.
	\end{rmk}

\end{section}

\begin{section}{Structure of branching in \titletex{$L$}}\label{sec:branching}
	Let $L$ be the leaf space of the lift $\hatF$ of $\F$ to $\widetilde{\MM}$. In this section, we roughly prove that ``all branching in $L$ happens in the saddle regions". These results are not logically required for the proof of the main theorem (and indeed do not hold in full generality), but are interesting in their own right and provide motivation for subsequent constructions.

	In this section, we make the following assumption:
	\begin{ass}\label{ass:hyperbolic}
		All the monkey saddles used in the construction of $\F$ have at least four sides.
	\end{ass}

	In the case where all the monkey saddles have two sides, the resulting manifolds carry Anosov flows. The branching in their invariant foliations was analyzed in \cite{fenley_structure_1998}.

	Recall that for a taut foliation, the leaf space of the universal cover is a simply connected (possibly non-Hausdorff) 1-manifold. Each leaf of the universal cover is homeomorphic to a plane, and its stabilizer under the action of $\pi_1$ by deck transformations is the fundamental group of the projected leaf in $\MM$. Two points in $L$ are said to be \emph{non-separated} if they coincide in the Hausdorffification of $L$. A \emph{branch locus} is maximal set of at least two points in $L$ that are pairwise non-separable. \cite[Chapter~4]{calegari_foliations_2007}

	\begin{prop}\label{prop:nobranching}
		Assume \cref{ass:hyperbolic}. For each orbit $K_i$ there are two branch loci called $B_i^+$ and $B_i^-$ (unique up to covering transformations) each of which is finite and has $p_i$ points corresponding with coherent lifts of the positively or negatively oriented walls of the filling saddle region. Moreover, these are all of the branch loci of $\hatF$ up to covering transformations. 
	\end{prop}

	\begin{proof}
		The core of a saddle region is a curve transverse to $\F$ and is therefore non-contractible. Therefore, the $S^1$ worth of saddle-like leaves in a saddle region lifts to an $\R$ worth of saddle-like leaves (plus translates thereof) in $\hatF$. This $\R$ worth of leaves limits to lifts of the positive (resp. negative) walls in the $+\infty$ (resp $-\infty$) direction, so the positively (resp. negatively) oriented walls of a saddle region form a branch locus. We choose one of the translates of this branch locus and call it $B_i^+$ (resp. $B_i^-)$. 

		To show that there is no branching elsewhere, it suffices to show that any curve can be ``pulled tight'' relative to $\F$ so that it is transverse to $\F$ except for controlled intervals in $\guts$. Roughly speaking, our strategy will be to pull $\gamma$ tight relative to the natural Solv metric on $\Lambda'$. 

		We say that an arc in $\guts$ with endpoints on the interior of the walls in $\bd \guts$ is \emph{inessential} if it can be compressed into a wall of $\bd \guts$, and \emph{essential} otherwise. An essential arc with endpoints in oppositely oriented walls can be homotoped in $\guts$ rel. boundary to be transverse to $\F$, while an essential arc with endpoints in similarly oriented walls can be homotoped in $\guts$ rel. boundary to have a single tangency with $\F$. An essential arc of the latter type lifts to a short curve in $L$ that connects two points in $B_i^\pm$. It suffices to show that any curve $\gamma$ with endpoints on leaves $\lambda_1, \lambda_2$ can be homotoped in $\MM$ relative to its endpoints so that either
		\begin{enumerate}
			\item $\gamma$ is transverse to $\F$ except for finitely many tangencies, each of which is contained in an essential subarc of $\gamma$ in $\guts$

				or

			\item $\lambda_1=\lambda_2$ and $\gamma$ lies inside $\lambda_1$.
		\end{enumerate}
		We call such a curve $\emph{efficient}$. Moreover, we can without loss of generality assume that $\lambda_1$ and $\lambda_2$ are walls of $\guts$. For, suppose that $\widetilde \lambda_1$ and $\widetilde \lambda_2$ are non-separated leaves of $\widetilde \F$ in a branch locus not equal to one of the $B_i^\pm$. We may choose $\widetilde \lambda_1$ and $\widetilde \lambda_2$ to be adjacent, meaning that there is a family of leaves $\{\widetilde{\lambda_t}\}_{t>0}$ limiting to both $\widetilde \lambda_1$ and $\widetilde \lambda_2$.\plabel{fix2:adjacent} Adjacency is a stronger condition that non-separatedness, but a pair of non-separated leaves can be connected by a sequence of adjacent leaves; we direct the reader to \cite[Example 4.44, page 170]{calegari_foliations_2007} for further discussion.\plabel{fix3:7} Without loss of generality, assume that $\widetilde \lambda_t$ approaches $\widetilde \lambda_1$ and $\widetilde \lambda_2$ from above in $L$ as $t\to 0$. Any path $\widetilde \gamma\subset \widetilde{\MM}$ from $\widetilde \lambda_1$ to $\widetilde{\lambda_2}$ descends to a path $\gamma$ in $\MM$ that cannot be homotoped in $\MM$ rel. endpoints to become efficient. Since $\F$ is taut, we can augment the beginning of $\gamma$ with a descending transversal from $\lambda_1$ to a wall of $\guts$. Similarly, we can augment the end of $\gamma$ with a descending transversal from $\lambda_2$ to a wall of $\guts$. This resulting path is also not homotopic in $\MM$ to an efficient one, but has endpoints on walls of $\guts$.\plabel{fix2:homotopy}

		Given any arc $\gamma$ with endpoints in leaves $\lambda_1,\lambda_2$ which are walls of $\guts$, we may homotope it in $\MM$ so that it intersects $\bd \guts$ transversely on the interior of the walls. Define
		\begin{itemize}
			\item $N_{\xtan}$ as the number of tangencies between $\gamma$ and $\F$ in $\MM \setminus \guts$.
			\item $N_{\xess}$ as the number of components of $\gamma \cap \guts$ that are essential.
			\item $N_{\xiness}$ as the number of components of $\gamma \cap \guts$ that are inessential.
		\end{itemize}

		Define the complexity function $$I(\gamma)=N_{\xess} + 1.01N_{\xtan} + 1.02N_{\xiness}.$$\plabel{fix2:complexity} Now choose $\gamma$ in its relative homotopy class to minimize $I(\gamma)$. With this choice, we claim that $N_{\xtan}=N_{\xiness}=0$.

		Suppose $N_{\xiness}>0$. Then we can compress an inessential arc in $\guts$. This decreases $N_{\xiness}$ by 1 at the cost of increasing $N_{\xtan}$ by one, violating the minimality assumption.

		Now suppose that $N_{\xtan}>0$. Then there is a component $\gamma_0$ of $\gamma \cap (\MM\setminus \guts)$ containing a tangency with $\F$. The suspension flow of $\phi$ blows up to a flow $X$ on $\MM\setminus\guts$ which is 
		\begin{enumerate}
			\item transverse to the fibering $\MM\setminus \guts \to S^1$,
			\item preserves $\Lambda$,
			\item and exits through the interstitial annuli (ie the annuli in $\partial \guts \cap \partial \interst$) in forward time and exists for all backwards time. 
		\end{enumerate}
		We can arrange that $\gamma_0$ lies in a single fiber surface $S$ of the fibering. We can do this by pushing $\gamma_0$ backwards along the flowlines of $X$ into some choice of fiber surface $S$.\plabel{fix2:flow} This possibly slides $\gamma_0(0)$ and $\gamma_0(1)$ along their respective walls of $\guts$, but does not change the complexity function. Using \cref{lem:lambdaprime}, we may replace $\gamma_0$ with an arc $\gamma_0'$ having the properties that
		\begin{enumerate}
			\item $\gamma_0'$ is in the support of $\Lambda'$ and
			\item $\gamma_0'$ is homotopic to $\gamma_0$ rel. endpoints and
			\item the number of tangencies of $\gamma_0$ with $\F$ is equal to the number of tangencies of $\gamma_0'$ with $\Lambda'$.
		\end{enumerate}
		
		The pseudo-Anosov structure of $\phi$ gives rise to a locally Euclidean Riemannian metric on the surface with boundary $\Lambda'\cap S$.\plabel{fix2:euclidean} Pull $\gamma_0'$ tight relative to this metric. Note that this tightening happens completely inside the surface $S$.\plabel{fix2:tighten} Since the leaves of $\Lambda'\cap S$ are geodesics (ie locally straight lines) with respect to this metric, the number of tangencies between $\gamma_0'$ and $\Lambda'$ does not increase during tightening. We have a couple cases:

		\begin{enumerate}
			\item $\gamma_0'$ tightens to a geodesic and is not contained in a single leaf of $\Lambda'$. Since the leaves of $\Lambda'\cap \Sigma$ are geodesic, $\gamma_0'$ cannot have any tangencies with leaves of $\Lambda'$. Thus, using \cref{lem:lambdaprime} we can also homotope $\gamma_0$ to be transverse to $\F$ and decrease $N_{\xtan}$.
			\item $\gamma_0'$ tightens to an arc contained in a wall of $\bd \guts'$. Note that distinct components of the guts have distinct boundary leaves, so this arc is really contained in a single wall of some component of $\guts'$. If $\gamma_0=\gamma$, then by the correspondence in \cref{lem:lambdaprime} we have successfully compressed $\gamma$ into a leaf and we're done. Otherwise, using the correspondence in \cref{lem:lambdaprime}, we can compress $\gamma_0$ into $\guts$ which reduces $N_{\xtan}$ by 1. If $\gamma_0$ shares an endpoint with $\gamma$, then $N_{\xess}+N_{\xiness}$ remains constant. See \cref{subfig:tighten2p}. Otherwise, $N_{\xess}+N_{\xiness}$ decreases by one (although each could individually increase). Therefore, the complexity function decreases. See \cref{subfig:tighten2} for an example in which $N_{\xtan}$ decreases from 1 to 0 and $N_{\xess}+N_{\xiness}$ decreases from 2 to 1.
			\item $\gamma_0'$ approaches $\guts'$ during tightening. Then there is a subarc of $\gamma_0'$ which wraps around $\bd \guts'$, making contact along a subarc of Euclidean angle $>\pi$. Walking around the singularity, one meets a prong every $\pi$ radians. Thus, the arc will cross at least two prongs of the singularity.\plabel{fix2:subarc} This arc must have a tangency to $\Lambda'$. Push $\pi$ of the corresponding subarc of $\gamma_0$ into $\guts$ so as to decrease $N_{\xtan}$ by exactly one. See \cref{subfig:tighten3}. We claim that the new arc created is essential so that $N_{\xess}$ increases by one. This fact crucially uses the assumption that the saddle region is an ideal $n$-gon bundle for some $n\geq 4$. A pair of walls of such a saddle region which are separated by two prongs must have distinct lifts in the universal cover of the solid torus saddle region. Therefore, the endpoints of the newly created arc in $\guts$ lie in different walls in this $\Z$-cover and so the new arc is essential. In total, the complexity function has decreased.
		\end{enumerate}
		\begin{figure}[ht]
			\begin{subfigure}{0.35\textwidth}
				\centering
				\def\svgwidth{0.85\linewidth}
				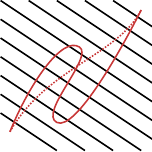
				\caption{}\label{subfig:tighten1}
			\end{subfigure}
			\begin{subfigure}{0.35\textwidth}
				\centering
				\def\svgwidth{0.85\linewidth}
				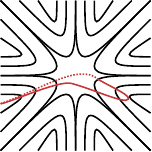
				\caption{}\label{subfig:tighten2p}
			\end{subfigure}
			\begin{subfigure}{0.35\textwidth}
				\centering
				\def\svgwidth{0.85\linewidth}
				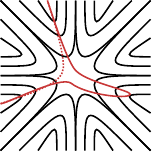
				\caption{}\label{subfig:tighten2}
			\end{subfigure}
			\begin{subfigure}{0.35\textwidth}
				\centering
				\def\svgwidth{0.85\linewidth}
				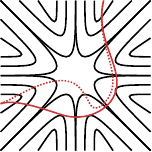
				\caption{}\label{subfig:tighten3}
			\end{subfigure}
			\caption{Various tightening moves}
		\end{figure}

		It follows that the minimizer $\gamma$ has $N_{\xtan}=N_{\xiness}=0$, so every tangency between $\gamma$ and $\F$ occurs inside an essential arc in $\guts$. Finally, we pull each component of $\gamma \cap \guts$ tight so as to contain either 0 or 1 tangencies with $\F$. The arc $\gamma$ is now efficient. \qedhere

	\end{proof}

	\begin{rmk}\label{rmk:anosovbranching}
		Although \cref{cnstr:foln} works more generally for manifolds with a pseudo-Anosov flow having orientable invariant foliations, \cref{prop:nobranching} does not hold. For example, there are many manifolds supporting Anosov flows whose invariant foliations have branching \cite{fenley_structure_1998}. The existence of a Solv metric on $\M\setminus \guts$ and the presence of ``negative curvature'' in $\guts$ from \cref{ass:hyperbolic} are crucial.
	\end{rmk}

	\begin{prop}
		Assume \cref{ass:hyperbolic}. The stabilizers of $B_i^+$ and $B_i^-$ in $\pi_1(\MM)$ are both equal to the infinite cyclic group $\Stab(B_i^\pm)$ generated by a conjugate to the Dehn filling core $\corei$. $\Stab(B_i^\pm)$ preserves a circular order on the points of $B_i^+$.
	\end{prop}
	\begin{proof}
		Let $R$ be the relevant saddle region. Observe that $\Stab(R)=\Stab(B_i^+)=\Stab(B_i^-)$. Every leaf in $R$ is homeomorphic to a plane, and so has trivial stabilizer. Therefore, there is at most one element of $\Stab(R)$ mapping between lifts $\lambda_1,\lambda_2 \in \hatF$ of a leaf $\lambda \in R$. There is always a power of $K_i(\textbf{p};\textbf{q})$ which accomplishes this transformation, so $\Stab(R)$ consists only of such elements. The circular order on walls is the order in which they appear as sides of a fiber of the ideal polygon bundle structure on $R$.
	\end{proof}

	Let us now recall a standard fact about periodic orbits of $\phi$.
	\begin{lem}\label{lem:orbitsandcylinders}
		The non-planar leaves of $\Lambda$ are all cylindrical. Moreover, they are in correspondence with primitive periodic orbits of $\phi$. This correspondence is $k$-to-one for the leaves intersecting $\bd \guts$, where $k$ is the number of walls of the incident component of $\guts$. It is one-to-one for all other leaves.
	\end{lem}
	\begin{proof}
		In what follows, we regard the various prongs of a singular leaf of $\F^s$ or $\F^s\cap \Sigma$ as separate leaves.\plabel{fix2:4} Since $\phi$ preserves no closed curve, the intersection of any leaf in $\F^s$ with $\Sigma$ is either a prong (homeomorphic to $[0,\infty)$) or a copy of $\R$. Leaves of $\F^s$ with nontrivial topology are suspensions of those leaves of $\F^s\cap \Sigma$ which are preserved by some power of $\phi$, and hence are homeomorphic to cylinders. Since $\phi$ is pseudo-Anosov, its action contracts lengths in each leaf of $\F^s\cap \Sigma$. A contraction on $[0,\infty)$ or $\R$ has a unique fixed point, so each cylindrical leaf of $\F^s$ contains exactly one primitive orbit of the suspension flow of $\phi$.

		Conversely, the suspension of a primitive periodic orbit of $\phi$ is either contained in a unique leaf of $\F^s$ or is a singular orbit contained in two or more prongs.

		We obtain $\Lambda$ from $\F^s$ by splitting open the leaves of $\F^s$ which contain orbits in $\set{K_i}$. This has the effect of doubling such leaves and then gluing any leaves of the form $[0,\infty)\times S^1$ in pairs; the conclusion of the lemma follows for the leaves intersecting $\bd \guts$. All other leaves are preserved in the passage from $\F^s$ to $\Lambda$.
	\end{proof}

	\begin{prop}\label{prop:nofixedpoints}
			Assume \cref{ass:hyperbolic}. No element of $\Stab(B_i^\pm)$ stabilizes any point in the Hausdorffification of $L$ aside from $B_i^+$ and $B_i^-$.
	\end{prop}
	\begin{proof}
		$\Stab(B_i^\pm)$ is generated by a loop freely homotopic to the Dehn surgery core $\corei$. The stabilizer of a cylindrical leaf in $\Lambda$ is represented by a loop freely homotopic to the suspension of a periodic point of $\phi$. We must show that there are no nontrivial free homotopies among loops $\corei$ (or powers thereof) or suspensions periodic orbits of $\phi$ (or powers thereof). Call the set of these curves $\mathcal O$. For simplicity, we replace each $\corei$ with a power thereof that is homotopic to one living inside a wall of $\bd \guts$. \plabel{fix:13}

		First, we rule out a nontrivial homotopy between two curves in $\mathcal O$ that stays inside $\MM \setminus \guts$. Now $\MM \setminus \guts$ embeds in the fibered manifold $\M$ so that the elements of $\mathcal O$ are suspensions of periodic orbits of $\phi$. Since $\phi$ is pseudo-Anosov, these suspensions are never homotopic to one another. 

		Now we must rule out a nontrivial homotopy between curves in $\mathcal O$ that might pass through the guts $\guts$. The idea is that the saddle regions contains lots of negative curvature, but an annulus giving rise to the purported homotopy has Euler characteristic zero and therefore can't \plabel{fix:14}cut across any saddle region. To be more precise, let $\gamma$ and $\eta$ be two elements of $\mathcal O$. Suppose $\gamma$ is homotopic to $\eta$ in $\MM$. Replacing $\gamma$ and $\eta$ by $\gamma^k$ and $\eta^k$ if necessary, we can assume that $\gamma$ and $\eta$ both lie in leaves of $\F$. Let $A=[0,1]\times S^1$ and choose an immersion $\iota:A \to \MM$ realizing a homotopy from $\gamma$ to $\eta$. \plabel{fix2:9:1} By Roussarie-Thurston, we can homotope the annulus so that the induced foliation has no critical points \cite{roussarie_plongements_1974, thurston_norm_1986}. We can also make $A$ transverse to $\bd \guts$ and the corners of $\bd \guts$.\plabel{fix2:9:2} Let $C=\iota^*(\bd \guts)$; by our transversality conditions, $C$ is a collection of disjoint closed polygonal curves in $A$. Choose $\iota$ to minimize number of components of $C$.

		\begin{enumerate}[label=\textbf{Case~\arabic*}]
			\item $C$ contains an innermost loop $\alpha$ that is inessential in $A$. Let $D$ be the disk in $A$ bounded by $\alpha$. The boundary of $D$ is a $2n$-gon with sides alternating between arcs transverse to the foliation and arcs tangent to the foliation. Here, $2n$ is the number of corners of $\partial G$ that $\partial D$ intersects. Moreover, the angles between adjacent arcs are all convex or all concave, depending on whether $\interior D$ lies inside or outside $\guts$.\plabel{fix2:9:3} See \cref{subfig:case1} for a picture of the convex case with $n=4$. The foliation on $D$ is oriented and has no critical points. An application of the Poincare-Hopf index theorem shows that $n=2\chi(D)=2$ in the convex case, or $n=-2\chi(D)=-2$ in the concave case which is impossible. Since every saddle region is an ideal polygon bundle with fiber having at least four sides, a meridian in $\bd \guts$ crosses at least eight corners. Therefore, $\alpha$ is not a meridian or a multiple of a meridian. Furthermore, $\alpha$ cannot be homotopic to a multiple of a core of $\guts$ since $\alpha$ is zero in $\pi_1(\MM)$. Therefore, $\alpha$ must be inessential in $\bd \guts$. We can then compress $D$ into $\bd \guts$ (while maintaining Roussarie-Thurston general position) and eliminate the curve $\alpha$ from $C$, contradicting our minimality assumption. 

				\begin{figure}[!htb]
					\begin{subfigure}{0.49\textwidth}
						\centering
						\def\svgwidth{0.99\linewidth}
\begingroup%
  \makeatletter%
  \providecommand\color[2][]{%
    \errmessage{(Inkscape) Color is used for the text in Inkscape, but the package 'color.sty' is not loaded}%
    \renewcommand\color[2][]{}%
  }%
  \providecommand\transparent[1]{%
    \errmessage{(Inkscape) Transparency is used (non-zero) for the text in Inkscape, but the package 'transparent.sty' is not loaded}%
    \renewcommand\transparent[1]{}%
  }%
  \providecommand\rotatebox[2]{#2}%
  \newcommand*\fsize{\dimexpr\f@size pt\relax}%
  \newcommand*\lineheight[1]{\fontsize{\fsize}{#1\fsize}\selectfont}%
  \ifx\svgwidth\undefined%
    \setlength{\unitlength}{181.41732283bp}%
    \ifx\svgscale\undefined%
      \relax%
    \else%
      \setlength{\unitlength}{\unitlength * \real{\svgscale}}%
    \fi%
  \else%
    \setlength{\unitlength}{\svgwidth}%
  \fi%
  \global\let\svgwidth\undefined%
  \global\let\svgscale\undefined%
  \makeatother%
  \begin{picture}(1,0.55031508)%
    \lineheight{1}%
    \setlength\tabcolsep{0pt}%
    \put(0,0){\includegraphics[width=\unitlength,page=1]{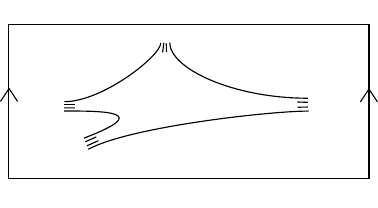}}%
    \put(0.43895396,0.51251593){\color[rgb]{0,0,0}\makebox(0,0)[lt]{\lineheight{1.25}\smash{\begin{tabular}[t]{l}$\gamma$\end{tabular}}}}%
    \put(0.45938098,0.02544799){\color[rgb]{0,0,0}\makebox(0,0)[lt]{\lineheight{1.25}\smash{\begin{tabular}[t]{l}$\eta$\end{tabular}}}}%
    \put(0.59208532,0.34153178){\color[rgb]{0,0,0}\makebox(0,0)[lt]{\lineheight{1.25}\smash{\begin{tabular}[t]{l}$\alpha$\end{tabular}}}}%
  \end{picture}%
\endgroup%

						\caption{}\label{subfig:case1}
					\end{subfigure}
					\begin{subfigure}{0.49\textwidth}
						\centering
						\def\svgwidth{0.99\linewidth}
\begingroup%
  \makeatletter%
  \providecommand\color[2][]{%
    \errmessage{(Inkscape) Color is used for the text in Inkscape, but the package 'color.sty' is not loaded}%
    \renewcommand\color[2][]{}%
  }%
  \providecommand\transparent[1]{%
    \errmessage{(Inkscape) Transparency is used (non-zero) for the text in Inkscape, but the package 'transparent.sty' is not loaded}%
    \renewcommand\transparent[1]{}%
  }%
  \providecommand\rotatebox[2]{#2}%
  \newcommand*\fsize{\dimexpr\f@size pt\relax}%
  \newcommand*\lineheight[1]{\fontsize{\fsize}{#1\fsize}\selectfont}%
  \ifx\svgwidth\undefined%
    \setlength{\unitlength}{181.41732283bp}%
    \ifx\svgscale\undefined%
      \relax%
    \else%
      \setlength{\unitlength}{\unitlength * \real{\svgscale}}%
    \fi%
  \else%
    \setlength{\unitlength}{\svgwidth}%
  \fi%
  \global\let\svgwidth\undefined%
  \global\let\svgscale\undefined%
  \makeatother%
  \begin{picture}(1,0.51836252)%
    \lineheight{1}%
    \setlength\tabcolsep{0pt}%
    \put(0,0){\includegraphics[width=\unitlength,page=1]{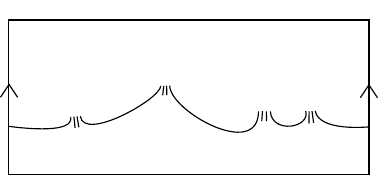}}%
    \put(0.52059131,0.489916){\makebox(0,0)[t]{\lineheight{1.25}\smash{\begin{tabular}[t]{c}$\gamma$\end{tabular}}}}%
    \put(0.53097714,-0.00234613){\makebox(0,0)[lt]{\lineheight{1.25}\smash{\begin{tabular}[t]{l}$\eta$\end{tabular}}}}%
    \put(0.52602707,0.21794847){\makebox(0,0)[lt]{\lineheight{1.25}\smash{\begin{tabular}[t]{l}$\alpha$\end{tabular}}}}%
  \end{picture}%
\endgroup%

						\caption{}\label{subfig:case2}
					\end{subfigure}
					\caption{The annulus $A$ is shown along with the induced foliation in a neighbourhood of $\alpha$ in two disallowed configurations. If $\alpha$ crosses the sutures in $\bd \guts$ as shown, there is no way to extend the foliation to the rest of $A$ without critical points, contradicting Roussarie-Thurston.}\label{fig:case2}
			\end{figure}

			\item $C$ contains only loops that are essential in $A$.\plabel{fix:16} Let $\alpha$ be any element of $C$. As in the first case, $\alpha \cap A$ alternates between arcs between transverse and tangent to $\F$, and the angles between adjacent arcs all either convex or concave depending on which side of $\alpha$ is on the interior of $\guts$. Let $n$ be the number of alternations between transverse and tangent arcs in $\alpha$. See \cref{subfig:case2} for the case $n=4$. Let $R$ be the annulus cobounded by $\gamma$ and $\alpha$. Since $\F\cap R$ has no singularities, we have $n=2\chi(R)=0$ and we must be in the convex case. Thus, $\alpha$ actually intersects no corners of $\bd \guts$ and is either entirely transverse to $\F$ (ie is contained in a suture of $\bd \guts$) or contained in a wall of $\guts$. Moreover, $\alpha$ is not contractible in $\partial G$ since $\alpha$ is homotopic in $\MM$ to $\gamma$ which is even non-contractible in $\MM$. Therefore, $\alpha$ is a closed curve parallel to the sutures in some component of $\bd \guts$.
			
			The arguments above hold for any $\alpha \in C$. The curves in $C$ cut $A$ into a collection of annuli which we label $A_1\dots A_M$. Each $A_k$ is either an annulus living in $\MM\setminus \guts$ or an annulus in $\guts$. In the first alternative, $A_k$ represents a free homotopy in $\MM\setminus\guts$; as noted at the beginning of the proof, the two boundary components of such an $A_k$ must represent the same element of $\mathcal O$. In the latter alternative, we can trivially make the same conclusion. Therefore, $\gamma$ and $\eta$ represent the same element of $\mathcal O$ as desired.\qedhere
		\end{enumerate}

	\end{proof}
\end{section}

\begin{section}{Gluing branches}\label{sec:gluing}
	In this section, we give some motivation for the constructions in \cref{sec:connection}. We wish to define on $L$ an equivalence relation, denoted $\fequiv$, such that $(L/ \fequiv) \cong \R$.

	\begin{subsection}{Hausdorff and non-Hausdorff quotients}	
	We first describe this gluing process in a simpler, abstract case where $L$ is replaced by an infinite binary tree $T$. Consider an infinite rooted binary tree $T$, each of whose vertices has a distinguished left and right child. We think of $T$ not combinatorially, but as a topological space. Let $v_0$ be the root vertex of $T$. Given a vertex $v\in T$, let $vL$ represent its left child and $vR$ represent its right child. We use the notation $\set{L,R}^k$ to denote the set of words of length $k$ in the two letter alphabet $\set{L,R}$\plabel{fix2:10}. We use fractional powers of $R$ to denote points that are on edges of the tree; for example, $vR^{\frac 1 2}$ is the midpoint of the path $[v,vR]$. Then any point on the tree is of the form $v_0\{L,R\}^k L^{\alpha}$ or $v_0\{L,R\}^k R^\alpha$ for some $k\geq 0$, $0\leq \alpha < 1$.

	In order to illustrate the kinds of difficulties we will encounter in the next section, we will show two superficially similar equivalence relations on $T$ such that the quotient is $[0,\infty)$ in the first but is non-Hausdorff in the latter. 

	\begin{enumerate}[label=\textbf{Example~\arabic*}]
		\item Let's define the first equivalence relation. At each vertex $v$, glue the infinite paths $[v,vLLL\dots)$ to $[v,vRRR\dots)$ together by the obvious map $[0,\infty) \to [0,\infty)$ which preserves depths of vertices. See \cref{fig:tree1}. This has the effect of collapsing the entire tree down to a copy of $[0,\infty)$. \plabel{fix:17}
	\begin{figure}[ht]
		\centering
		\def\svgwidth{0.8\linewidth}
		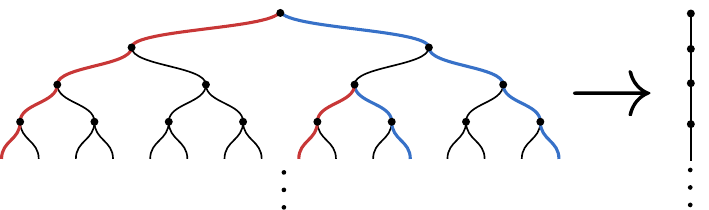

		\caption{The red paths are of the form $[v,vLLL\dots)$ and the blue paths are of the form $[v,vRRR\dots)$. Points which are equivalent with respect to the gluing are marked with the same number.}\label{fig:tree1}
	\end{figure}
	\item Let us construct a different equivalence relation such that the quotient is not Hausdorff. At each vertex $v\in T$ we glue the infinite descending paths $[v,vLRRR\dots)$ and $[v,vRLRRR\dots)$ by a homeomorphism which sends $[v,vL]$ to $[v,vRL]$ by a dilation by a factor of two, and sends $[vL,vLRRR\dots)$ to $[vRL,vRLRRR\dots)$ by an isometry. See \cref{fig:tree2}. The points $v_0R^nL$ are identified for all $n\geq 0$. Therefore, $[v_0,v_0RRR\dots)$ is not properly embedded in the quotient space. It turns out that the quotient space is not $[0,\infty)$ but a tree with infinite valence at each vertex. For example, the points $v_0LL$ and $v_0RLL$ are not comparable in the partial order induced on the quotient. For points $a,b\in T$, let $[a]$ and $[b]$ denote their images in the quotient. Write $a \geq b$ if $a$ is an ancestor of $b$ in $T$, and $[a]\geq[b]$ if $[a]$ is an ancestor of $[b]$ in the quotient. The interested reader is invited to show the following, which give a precise picture of the quotient of $T$ with respect to our second gluing.
	\begin{enumerate}\plabel{fix:12}
		\item Each equivalence class under the gluing has a canonical representative of the form $v_0\{L,R\}^k LLR^{s}$, $v_0LR^s$, or $v_0R^s$ for some $s\geq 0$.
		\item Suppose $a$ and $b$ are canonical representatives of their equivalence classes. Write $a=rWR^s$ for some word $W$ in $L$ and $R$, and $s\in \R$ maximal. Then $[a] \geq [b]$ if and only if one of the following holds:
		\begin{enumerate}
			\item $a\geq b$ 
			\item $W$ is the empty word and $v_0WR^nL\geq b$ for some integer $n\leq s$ 
			\item $W=W'L$ and $v_0W'R^kLL \geq b$ for some integer $k\geq 0$
		\end{enumerate}\plabel{fix2:12:1}
	\end{enumerate}
	\end{enumerate}	
	\plabel{fix2:11}
	\begin{figure}[ht]
		\centering
		\def\svgwidth{\linewidth}
		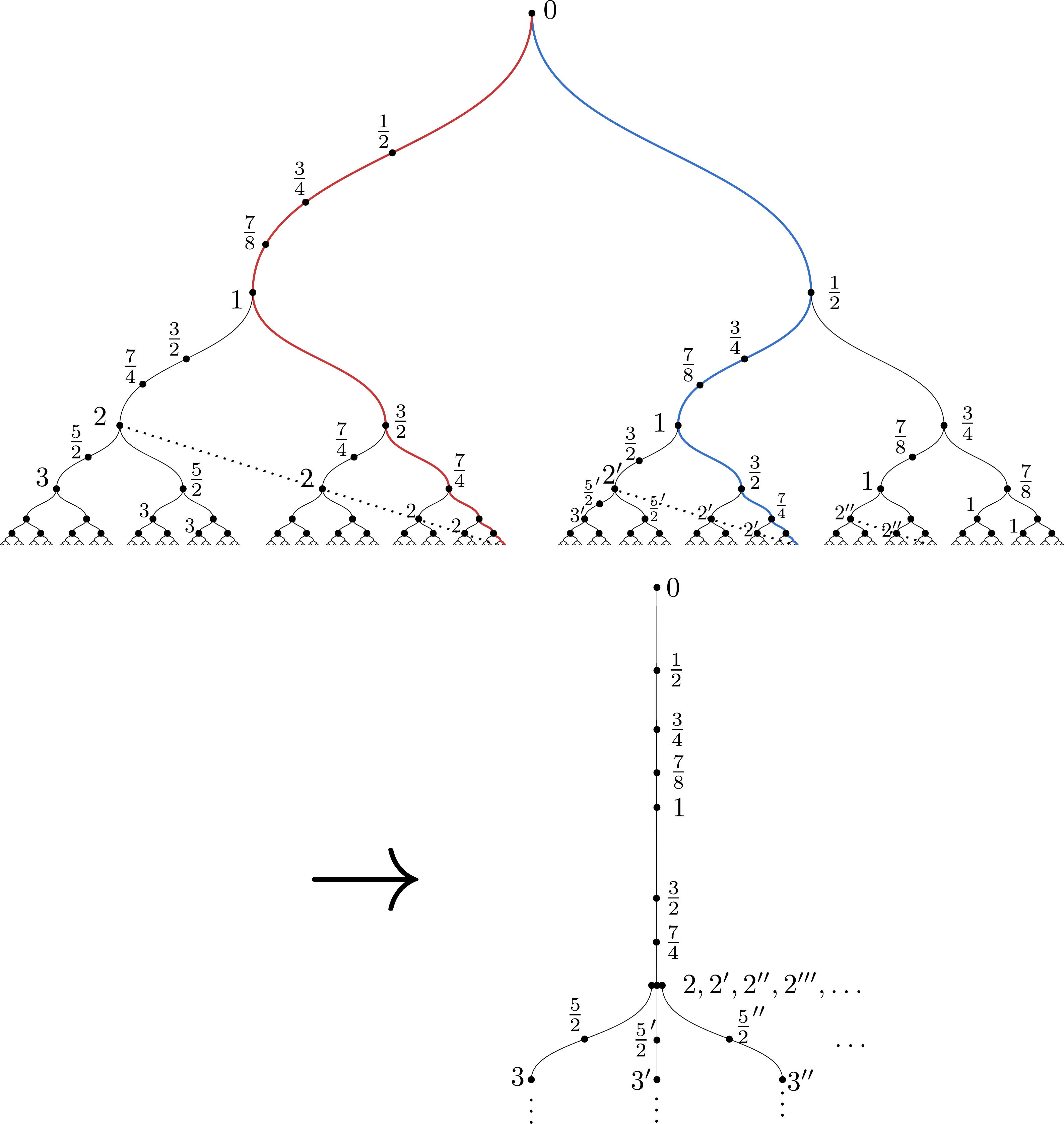
		\caption{Two paths that are to be glued together, $[v_0,v_0LRRR\dots)$ and $[v_0,v_0RLRRR\dots)$, are shown in red and blue respectively. Dotted lines are drawn through the equivalence classes $2$, $2'$, and $2''$. These classes map to non-separable points in the quotient. Their canonical representatives are $v_0LL$, $v_0RLL$, and $v_0RRLL$. For any $s\geq 0$, the point $v_0LR^s$ (marked with the label $2-\frac{1}{2^s}$) is a common ancestor of all three of these points in the quotient. Moreover, as $s\to \infty$, $v_0LR^s$ converges to all three of these points, so the space is not Hausdorff.}\label{fig:tree2}
	\end{figure}
	
	\end{subsection}
	\begin{subsection}{Gluing branches in $L$}
	\newcommand{\pathx}{[x,\infty)}
	\newcommand{\openpathx}{(x,\infty)}
	We wish to perform a similar gluing on the leaf space $L$. The leaf space should be thought of as a kind of tree, but possibly with a dense set of branching points.\plabel{fix:18} There are two new features in this case. First, there may be branching in both the upward and downward directions. Second, we must do this gluing in a $\pi_1$-equivariant way so that the action of $\pi_1$ will descend to the quotient $L/\fequiv$. The gluings we performed in the case of the binary tree are equivariant with respect to the semigroup of isometric self-maps of the binary tree which preserve the left and right children at each node. More care will be required in $L$ since the stabilizer of a branch point may be nontrivial.
	
	Let $b_i$ be a generator for $\Stab(B_i^\pm)$. Let $a_i$ be the smallest positive integer such that $b_i^{a_i}$ stabilizes $B_i$ pointwise. In other words, $a_i$ is the integer such that $b_i^{a_i}$ is freely homotopic to the essential loop in a wall of $\guts$. We fix the orientation of $b_i$ by asking that the pushoff of $b_i^{a_i}$ to the outside of $\guts$ has negative intersection number with $\Sigma$. With this choice, the holonomy of $\F$ on the outside of $\guts$ along a curve freely homotopic to $b_i^{a_i}$ is always attracting.\plabel{fix2:14} This implies that the action of ${b_i}^{a_i}$ on the leaf space is locally expanding. \plabel{fix:20}

	For each $x\in B_i^+$ (ie a lift of a wall of a saddle region), let $\openpathx \subset L$ be the lift to $L$ of the transverse loop obtained by pushing the essential loop in $x$ slightly outside the saddle region corresponding to $B_i^+$. Another way to think of $\openpathx$ is as the path in $L$ correponding with the set of leaves of $\hatF$ intersecting a prong of the unstable invariant foliation $\F^u$. Let $\pathx =  \set{x} \cup \openpathx$. The paths $\set{\pathx \mid x\in B_i^+}$ will play the role of the paths $[v,vRRR\dots)$ or $[v,vLLL\dots)$ in the case of the binary tree.
		
	\begin{lem}\label{lem:fixedpath}
		$\pathx$ is a $b_i^{a_i}$ invariant path from $x$ to $\infty$ in $L$. Furthermore, $b^{a_i}$ acts on $\pathx$ by an expanding dilation fixing $x$.
	\end{lem}
	\begin{proof}
		Since $\openpathx$ is a lift of a loop homotopic to $a_i$ times the core of the corresponding saddle region, $\pathx$ is $b_i^{a_i}$ invariant. Suppose $\pathx$ had a greatest lower bound $y$. Then either $y$ and $b^{a_i}y$ are distinct non-separated leaves or $y=b^{a_i}y$.\plabel{fix2:13} By \cref{prop:nofixedpoints}, this never happens. Therefore, $\pathx$ is properly embedded in $L$. Since the holonomy around the essential loop in $x$ is repelling on the outside of $\guts$ and $b_i^{a_i}$ has no fixed points in $\pathx$ besides $x$, it must be that $b_i^{a_i}$ acts by a homeomorphism conjugate to a dilation with stretch factor $> 1$.
	\end{proof}
	
	Call $[x,\infty)$ the \emph{invariant path} at $x$. We are now ready to give a first attempt at defining the equivalence relation $\sim$ on $L$.

	\begin{cnstr}\label{cnstr:gluing}\plabel{fix:21}
		Recall that $B_i^+$ is an arbitrarily chosen representative among its translates by deck transformations.\plabel{fix2:15:1} For each point $x$ in the branch locus $B_i^+$ (resp. $B_i^-$), construct the upward (resp. downward) oriented path $\pathx\subset L$ from $x$ to $+\infty$ (resp. $-\infty$) which is invariant under $b_i^{a_i}$ as in \cref{lem:fixedpath}. Now $b_i$ acts on $\bigcup_{x\in B_i^+} \pathx$. We shall now glue together the various paths in the set $\set{\pathx\mid x\in B_i^+}$. Up to reparameterizing $\pathx$, we can assume that $b_i^{a_i}$ acts on each $\pathx$ by dilation by a factor $2^{a_i}$. For each $x\in B_i^+$, each $y\in \pathx$ and each $m\in \Z$, declare $y \fequiv \frac 1 {2^m} b_i^m y$. This has the effect of gluing together the paths $\set{\pathx \mid x\in O}$ for $O\subset B_i^+$ an orbit of the $\Stab(B_i^+)$ action on $B_i^+$. See \cref{fig:branches}. The action of $b_i$ descends to the quotient $L/\fequiv$. In the quotient, $b_i$ acts on $\pathx/\fequiv$ as a dilation by a factor of 2.
		
		We have now glued together $[x_i,\infty)$ and $[x_j\infty)$ whenever $x_i,x_j$ are points of $B_i^+$ lying in the same orbit of the $\Stab(B_i^+)$ action. However, the $\Stab(B_i^+)$ action on $B_i^+$ may not be transitive. For example, consider the case when $K_i$ is an untwisted orbit with $2k$ prongs. \xlabel{fix6:1}{In the nonsingular case $k=1$, we have $|B_i^+|=|B_i^-|=1$.} In the singular case $k>1$, the action of $b_i$ on $B_i^+$ is not transitive because it does not permute the $k$ points in $B_i^+$.\plabel{fix5:4} We also wish to additionally glue those paths when $x_i$ and $x_j$ are in different orbits. Moreover, we wish to do this in some $b_i$ equivariant way. Observe that the induced action of $b_i$ on the quotient is the same on $[x_i,\infty)/\fequiv$ and $[x_j,\infty)/\fequiv$; both induced actions are expanding dilations. Thus, when $x_i$ and $x_j$ lie in different orbits, we simply choose some gluing map from $[x_i,\infty)$ to $[x_j,\infty)$ that intertwines their respective $\Stab(B_i^+)$ actions. \plabel{fix3:11} \plabel{fix2:16} 
		
		In total, we have glued together $|B_i^+|$ semi-infinite paths from $B_i^+$ to $\infty$ together in a $\Stab(B_i^\pm)$ equivariant way. Perform this gluing procedure for $B_i^-$ as well. Finally, we perform a gluing at each of the translates of $B_i^\pm$ in the way dictated by $\pi_1$-equivariance. This concludes the construction.\plabel{fix2:15:2}
	\end{cnstr}
	\begin{figure}[ht]
		\centering
		\def\svgwidth{0.65\linewidth}
		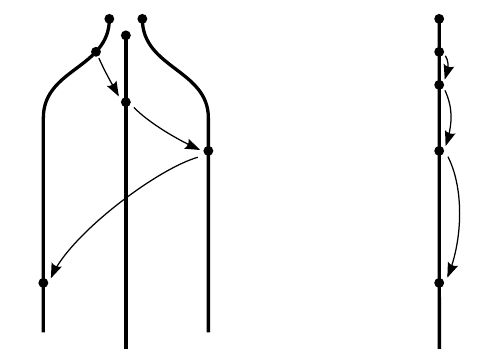
		\caption{On the left, the action of $b_i$ on $\bigcup_{x\in O} [x,\infty)$ is shown for the case $a_i=3$. Here, $O=\set{x_1,x_2,x_3}\subset B_i^-$. The action of $b_i$ descends to the quotient by $\fequiv$. This action on the quotient is shown on the right.}\label{fig:branches}
	\end{figure}
	
	Note that if $x\fequiv y$, then $gx \fequiv gy$ for any $g\in \pi_1(\MM)$. Therefore, \cref{cnstr:gluing} gives a $\pi_1$-equivariant equivalence relation $\fequiv$ on leaves. Unfortunately, $L/\fequiv$ is not homeomorphic to $\R$ due to a phenomenon like that of \cref{fig:tree2}. In \cref{sec:connection}, we present a coarsening of this equivalence relation which finishes the job.
	\end{subsection}
\end{section}

\begin{section}{The flat connection \titletex{$\widehat \J$}}\label{sec:connection}
	We prove the main theorems in this section. Recall that $\Mpunc=\M\setminus\set{K_i}$. Our first goal is to construct a trivial $S^1$-bundle (with structure group $\Homeo^+(S^1)$) on $\Mpunc$ with a flat connection $\J$ whose monodromy around the filling curves on the toroidal ends is trivial. This $S^1$ bundle will lift to an $\R$-bundle with a flat connection $\widehat J$ having the same property.

	For the purposes of defining $\J$ and proving \cref{thm:main}, we will need only the singular foliations $\F^u$ and $\F^s$. The foliation $\F$ will figure in the proof of \cref{thm:main2} where we make use of the fact that $\F$ blows down to $\F^s$.\plabel{fix2:51}

	\begin{subsection}{Preliminaries on connections and partial connections}\label{sec:connectionintro}
		We will use the following convention throughout the rest of the paper.
		\begin{conv}[Concatenation of paths]\label{conv:concat}
			Let $X$ be a topological space. Given paths $\gamma_1$ and $\gamma_2$ in $X$, \xlabel{fix6:2.1}{we define their concatenation $\gamma_1 * \gamma_2$ to be the path which first follows $\gamma_1$, then follows $\gamma_2$.} We use $*$ as multiplication in $\pi_1(X)$. With this convention, the action of $\pi_1(X)$ on $\widetilde X$ by deck transformations is determined as follows. Choose a basepoint $x_0\in \widetilde X$. For any $g\in \pi_1(M)$ and $q\in \widetilde{X}$, let $\gamma_1$ be a path representing $g$ and let $\gamma_2$ be a path in $X$ that lifts to a path from $x_0$ to $q$. Then $\gamma_1 * \gamma_2$ lifts to a path from $x_0$ to $gq$.
		\end{conv}
		Let $B$ be a topological space\plabel{fix2:17}. Given a topological bundle $F\to E \xrightarrow{\pi} B$, a \emph{connection} $\H$ is a choice of a homeomorphism $\H_\gamma\colon \pi^{-1}(\gamma(0))\to \pi^{-1}(\gamma(1))$ for each continuous path $\gamma\colon [0,1]\to B$. These homeomorphisms are required to be independent of the parameterization of $\gamma$ and to satisfy functoriality conditions with respect concatenation of paths. We refer to $\H_\gamma$ as \emph{parallel transport} along $\gamma$ with respect to $\H$. The functoriality conditions are:

		\begin{enumerate}
			\item $\H_{\gamma * \lambda} = \H_{\lambda} \circ \H_{\gamma} $ where $\lambda$ and $\gamma$ are concatenable paths in $B$ and $\gamma * \lambda$ is their concatenation.
			\item $\H_{\gamma * \gamma^{-1}}=id$, where $\gamma^{-1}$ is $\gamma$ traversed backwards.
			\item parallel transport along a trivial path is the identity map.
		\end{enumerate}

		Said in a different way, let $\mathcal C$ be the category whose objects are topological spaces homeomorphic to $F$ and whose morphisms are homeomorphisms. An $F$-bundle over $B$ assigns to each point $x\in B$ the fiber over $x$ which is an object in $\mathcal C$. A connection is an extension of this assignment to a functor from the groupoid of paths in $B$ to $\mathcal C$. This perspective will be useful because we will sometimes specify a connection by defining it on a set of generators for the groupoid of paths in $B$. We might also replace the groupoid of paths in $B$ with a slightly larger but equivalent groupoid.

		\begin{rmk}
			For the purposes of this paper, we will not need to ask that $\H_\gamma$ is close to the identity when $\gamma$ is a short path. Thus, our definition of a connection makes sense even in the absence of a local product structure on the bundle. This is useful because we will actually use a connection to define the local product structure.
		\end{rmk}

		A connection is called \emph{flat} if $\H_\gamma=id$ whenever $\gamma$ is contractible. A flat connection gives rise to \xlabel{fix6:2.2}{a homomorphism of $\pi_1(B)$ into $\Homeo(F)$ defined by $[\gamma] \mapsto \H_\gamma^{-1}$ for $[\gamma] \in \pi_1(B)$.}

		A section of $\pi$ is called \emph{flat} with respect to $\H$ or $\H$-flat if $\H_\gamma(s(\gamma(0)))=s(\gamma(1))$ for all paths $\gamma \colon [0,1]\to B$. When $\H$ is understood, we will suppress mentioning it and simply say that a section is flat. 
		
		We may also refer to sections of $\pi$ over a curve $\gamma\colon [0,1]\to B$; this simply means a section of the pullback bundle over $[0,1]$. We will usually express such a section as a map $s\colon [0,1]\to E$. A flat section over a path is also called a \emph{parallel section}. For any path $\gamma\colon[0,1]\to B$ and $x\in \pi^{-1}(\gamma(0))$, there is a unique parallel section $s_x$ over $\gamma$ satisfying $s_x(0)=x$. It is traced out by parallel transport of $x$ along $\gamma$. In equations, this means $s_x(t)=\H_{\gamma\mid_{[0,t]}}(x)$ where $\gamma \mid_{[0,t]}$ is the restriction of $\gamma$ to $[0,t]$.

		A \emph{partial connection} is similar to a connection except that the homeomorphisms $\H_\gamma$ may not be defined on all of $\pi^{-1}(\gamma(0))$. A partial connection $\H$ is a choice of a homeomorphism $\H_\gamma$ for each path $\gamma\colon [0,1]\to B$ from some (topological) subspace of $\pi^{-1}(\gamma(0))$ to some subspace of $\pi^{-1}(\gamma(1))$. The subspaces may depend on the path $\gamma$. The homeomorphisms are again required to be independent of the parameterization of $\gamma$ and functorial with respect to composition of paths: 
		\begin{enumerate}		
			\item $\H_{\gamma * \lambda}(x) = \H_{\lambda} (\H_{\gamma}(x))$ for all $x\in \pi^{-1}(\gamma(0))$ at which the right side is defined. In particular, $\H_{\lambda * \gamma}(x)$ is defined whenever the right side is defined.
			\item $\H_{\gamma * \gamma^{-1}}(x)=x$ for all $x\in \pi^{-1}(\gamma(0))$ at which $\H_\gamma$ is defined. In particular, $\H_{\gamma * \gamma^{-1}}(x)$ is defined whenever $\H_\gamma(x)$ is defined.\plabel{fix3:12}
			\item Parallel transport along a trivial path is defined on the entire fiber $\pi^{-1}(\gamma(0))$ and is equal to the identity map.
		\end{enumerate}
		A partial connection is called \emph{flat} if for each point $x\in B$ and each compact set $W\subset \pi^{-1}(x)$, there exists a neighbourhood $U$ of $x$ such that for every contractible path $\gamma\colon [0,1] \to U$, $\H_\gamma$ is defined on all of $W$ and agrees with the identity map.

		\begin{rmk}		
		In contrast with the case of flat connections, the monodromy of a flat partial connection around a long, contractible loop need not agree with the identity map on its domain of definition.
		\end{rmk}
		
		A section is called \emph{flat} with respect to a partial connection $\H$ if $\H_\gamma(s(\gamma(0)))=s(\gamma(1))$ for all paths $\gamma:[0,1]\to B$ along which $\H_\gamma$ is defined on $s(\gamma(0))$. Define \emph{flat} or \emph{parallel} sections over paths analogously to the case of connections. Given a path $\gamma\colon [0,1]\to B$ and a point $x\in \pi^{-1}(\gamma(0))$, one may attempt to define a parallel section $s_x$ over $\gamma$ by $s_x(t)=\H_{\gamma\mid_{[0,t]}}(x)$. However, the right side may fail to be defined for large $t$. In this case, let $t_{\max{}}(x)=\sup \set{t \mid \H_{\gamma\mid_{[0,t]}}(x) \text{ is defined}}$. Then we say that the parallel section $s_x$ \emph{blows up} at time $t_{\max{}}(x)$.

		A flat partial connection is roughly the same thing as a foliated bundle, though we will not use that language since the total space of our bundle may not apriori have a topology making it a manifold.
		
		\begin{ex}
			A partial connection $\H$ on the bundle $(0,1)\to (0,1)\times \R \to \R$ may be defined as follows. For any $\gamma$ parameterizing a curve in $\R$ that is monotonically increasing or decreasing, $\H_{\gamma}$ is a homeomorphism between an open subinterval of the fiber over $\gamma(0)$ and an open subinterval of the fiber over $\gamma(1)$. We define $\H_{\gamma}(x)=x+\gamma(0)-\gamma(1)$ for any $x$ satisfying $x\in (0,1)$ and $x+\gamma(0)-\gamma(1)\in (0,1)$.\plabel{fix2:18}\plabel{fix:25} For $\gamma$ which is not monotonically increasing or decreasing, $\H_\gamma$ is defined by composition of monotonic paths. In this case, the range and domain of $\H_{\gamma}$ will be smaller than that specified in the formula. See \cref{fig:partialconnection}.
			\begin{figure}[ht]
				\centering
				\def\svgwidth{80mm}
\begingroup%
  \makeatletter%
  \providecommand\color[2][]{%
    \errmessage{(Inkscape) Color is used for the text in Inkscape, but the package 'color.sty' is not loaded}%
    \renewcommand\color[2][]{}%
  }%
  \providecommand\transparent[1]{%
    \errmessage{(Inkscape) Transparency is used (non-zero) for the text in Inkscape, but the package 'transparent.sty' is not loaded}%
    \renewcommand\transparent[1]{}%
  }%
  \providecommand\rotatebox[2]{#2}%
  \newcommand*\fsize{\dimexpr\f@size pt\relax}%
  \newcommand*\lineheight[1]{\fontsize{\fsize}{#1\fsize}\selectfont}%
  \ifx\svgwidth\undefined%
    \setlength{\unitlength}{187.8346405bp}%
    \ifx\svgscale\undefined%
      \relax%
    \else%
      \setlength{\unitlength}{\unitlength * \real{\svgscale}}%
    \fi%
  \else%
    \setlength{\unitlength}{\svgwidth}%
  \fi%
  \global\let\svgwidth\undefined%
  \global\let\svgscale\undefined%
  \makeatother%
  \begin{picture}(1,0.45138405)%
    \lineheight{1}%
    \setlength\tabcolsep{0pt}%
    \put(0,0){\includegraphics[width=\unitlength,page=1]{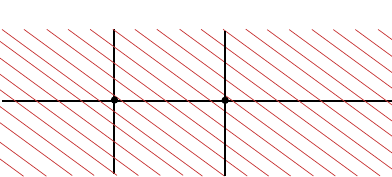}}%
    \put(0.2909813,0.4070427){\makebox(0,0)[t]{\lineheight{1.25}\smash{\begin{tabular}[t]{c}$\pi^{-1}(\gamma(0))$\end{tabular}}}}%
    \put(0.57557825,0.40762949){\makebox(0,0)[t]{\lineheight{1.25}\smash{\begin{tabular}[t]{c}$\pi^{-1}(\gamma(1))$\end{tabular}}}}%
  \end{picture}%
\endgroup%

				\caption{The total space for the bundle is shown with the fiber direction vertical. Parallel sections for the partial connection $\H$ are shown in red. Each parallel section is defined only for a limited time.}\label{fig:partialconnection}
			\end{figure}
		\end{ex}
	\end{subsection}

	\begin{subsection}{The partial connection \titletex{$\Jpar$}}\label{subsec:partial}
		Let $\Sigmapunc$ be a typical fiber of $\Mpunc$, ie $\Sigma$ punctured at its intersections with the $K_i$'s. $\Mpunc$ comes equipped with transverse stable and unstable codimension 1 foliations $\F^{s}$ and $\F^{u}$.\plabel{fix:23} It also has a natural incomplete Solv metric $g$ locally expressible as $dt^2 + \lambda^t dx^2 + \lambda^{-t} dy^2$, where $x$ and $y$ are local coordinates on $\Sigmapunc$, $\lambda>1$ is the stretch factor of $\phi$, and $t$ is the coordinate transverse to $\Sigmapunc$. By the orientability constraints, we may consistently establish cardinal directions north, south, east, and west. We choose our orientations on $\F^s$ and $\F^u$ so that $\F^s\cap \Sigmapunc$ is oriented eastward and $\F^{u}\cap \Sigmapunc$ to be oriented northward.\plabel{fix2:28:3} For us, a \emph{rectangle} will always mean a rectangle in a fiber surface free of singularities on its interior with top and bottom parallel to the east-west direction and left and right sides parallel to the north-south direction. Recall that we constructed $\F$ by splitting open $\F^s$. Let $\Sigmapunc_p$ be the fiber surface containing a point $p$. 
		
		A \emph{generalized leaf} of $\F^u\cap \Sigmapunc$ is one of:
		\begin{enumerate}
			\item a leaf of $\F^u \cap \Sigmapunc$ which is not a prong, or
			\item the concatenation of two prongs of $\F^u \cap \Sigmapunc$ which meet at the same singularity and make an angle of $\pi$ (ie a limit of leaves of $\F^u \cap \Sigmapunc$ approaching the singularity).
		\end{enumerate}
		
		We may similarly define a \emph{generalized leaf} of $\F^u$. \plabel{fix:dummy}The concatenation of two prongs $P_1$ and $P_2$ of $\F^u \cap \Sigmapunc$\plabel{fix2:20} incident with the same singularity $q$ formally requires an extra dummy point since $q$ itself is not a point in $\Mpunc$. We usually call the dummy point $q^*$ so that the generalized leaf is $P_1 \cup \set{q^*} \cup P_2$. If there is more than one generalized leaf in play, we will use $q^{**}, q^{***}$, etc to denote their dummy points.

		A point on a prong of $\F^u$ is contained in exactly two generalized leaves, but it would be better if each point were contained in exactly one generalized leaf. Thus, we will formally double each point $p$ on a prong $P$ of $\F^u$ into two points $p_{\xleft}$ and $p_{\xright}$. The points $p_{\xleft}$ and $p_{\xright}$ should be regarded as points in $\Mpunc$ infinitesimally perturbed to the west and east of $P$ respectively. Let $Z$ be the space obtained from $\Mpunc$ by doubling each point on a prong of $\F^u$. The open sets of the topology on $Z$ are the pullbacks of open sets in $\Mpunc$ under the obvious projection. Note that $Z$ is not Hausdorff and not even $T_0$ because for any point $p$ on a prong, $p_{\xleft}$ and $p_{\xright}$ are not topologically distinguishable.\plabel{fix2:21} In $Z$, each point is contained in exactly one generalized leaf of $\F^u$. One can construct a path in $Z$ from $p_{\xleft}$ to $p_{\xright}$ whose image is the two point set $\set{p_{\xleft}, p_{\xright}}$; such a path should be interpreted as an infinitesimal path in $\Mpunc$ crossing $P$ from $p_{\xleft}$ to $p_{\xright}$. As a consequence, $Z$ is path connected.\plabel{fix5:5}

		\begin{figure}[ht]
			\centering
			\def\svgwidth{80mm}
\begingroup%
  \makeatletter%
  \providecommand\color[2][]{%
    \errmessage{(Inkscape) Color is used for the text in Inkscape, but the package 'color.sty' is not loaded}%
    \renewcommand\color[2][]{}%
  }%
  \providecommand\transparent[1]{%
    \errmessage{(Inkscape) Transparency is used (non-zero) for the text in Inkscape, but the package 'transparent.sty' is not loaded}%
    \renewcommand\transparent[1]{}%
  }%
  \providecommand\rotatebox[2]{#2}%
  \newcommand*\fsize{\dimexpr\f@size pt\relax}%
  \newcommand*\lineheight[1]{\fontsize{\fsize}{#1\fsize}\selectfont}%
  \ifx\svgwidth\undefined%
    \setlength{\unitlength}{239.42173995bp}%
    \ifx\svgscale\undefined%
      \relax%
    \else%
      \setlength{\unitlength}{\unitlength * \real{\svgscale}}%
    \fi%
  \else%
    \setlength{\unitlength}{\svgwidth}%
  \fi%
  \global\let\svgwidth\undefined%
  \global\let\svgscale\undefined%
  \makeatother%
  \begin{picture}(1,0.78167677)%
    \lineheight{1}%
    \setlength\tabcolsep{0pt}%
    \put(0,0){\includegraphics[width=\unitlength,page=1]{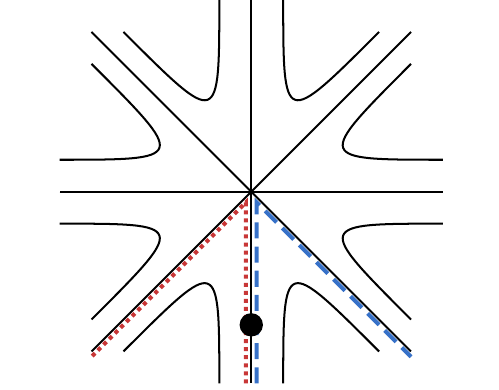}}%
    \put(0.52887698,0.15451902){\makebox(0,0)[lt]{\lineheight{1.25}\smash{\begin{tabular}[t]{l}$p$\end{tabular}}}}%
    \put(0.82557119,0.0108354){\makebox(0,0)[t]{\lineheight{1.25}\smash{\begin{tabular}[t]{c}$\theta\circ\pi^{-1}(p_{\rm east})$\end{tabular}}}}%
    \put(0.17913982,0.0148655){\makebox(0,0)[t]{\lineheight{1.25}\smash{\begin{tabular}[t]{c}$\theta\circ\pi^{-1}(p_{\rm west})$\end{tabular}}}}%
  \end{picture}%
\endgroup%

			\caption{The unstable foliation near a singularity with $k_i=4$. The fibers over $p_{\xleft}$ and $p_{\xright}$ are each unions of two prongs.}\label{fig:fibers}
		\end{figure}

		For any $p\in Z$, let $\Sigmapunc_p$ be the fiber surface containing $p$. (Here we are abusing notation slightly by using $\Sigmapunc_p$ to refer to its lift to $Z$. Let
		\begin{equation*}
		\begin{split}
			E_\pi=\{ (p,y) \mid  p&,y\in Z \text{ and }\\
								&\text{\(y\) is on the generalized leaf of $\F^u\cap \Sigmapunc_p$ which contains \(p\)}\}.
		\end{split}
		\end{equation*}\plabel{fix2:22:1}
		Let $\pi\colon E_\pi \to Z$ be the projection map $\pi\colon (p,y)\mapsto p$. Then $\pi$ defines an $\R$ bundle over $Z$. Intuitively, one should think of $\pi$ as the $\R$-bundle over $\Mpunc$ whose fiber over a point $p\in \Mpunc$ is the generalized leaf of $\F^{u} \cap \Sigmapunc_p$ containing $p$.

		Define the auxiliary map $\theta\colon E_\pi \to Z$ by $(p,y)\mapsto y$. In practice, we visualize points in $E_\pi$\plabel{fix2:22:2} via their images under $\theta$.\plabel{fix:theta}
		
		Note that there is no obvious local product topology on $E_\pi$ since the leaves of $\F^{u}$ diverge around every singular point. We will eventually define a topology by using a flat connection in \cref{sec:fullconnection}.

		In the rest of this section, when there is no chance of confusion, we will conflate $Z$ and $\Mpunc$. This means that we will think of $p_{\xleft}$ and $p_{\xright}$ as points in $\Mpunc$, think of $\pi$ as an $\R$-bundle over $\Mpunc$, and think of $\theta$ as a map from $E_\pi$ to $\Mpunc$. We will also think of the path from $p_{\xleft}$ to $p_{\xright}$ in $Z$ as an infinitesimal path in $\Mpunc$ with distinct endpoints. Since $\Mpunc$ embeds in $Z$ (for example by always choosing $p\mapsto p_{\xleft}$ when $p$ lies on a prong)\plabel{fix2:23}, a connection on $Z$ restricts to a connection on $\Mpunc$. Thus, we only gain generality by working in $Z$.

		\plabel{fix3:i}For a curve $\gamma:[0,1]\to \Mpunc$, there is a partially defined map $$h_\gamma \colon \theta\circ\pi^{-1}(\gamma(0))\to \theta \circ \pi^{-1}(\gamma(1))$$ which we refer to as the holonomy of $\F^s$ along $\gamma$. To define $h_\gamma$, choose a lift of $\gamma$ to the leaf space of $\widetilde \F^s$. (Here, $\widetilde {\F^s}$ is a foliation on $\widetilde \Mpunc$, obtained by lifting the restriction of $\F^s$ to $\Mpunc$ \plabel{fix3:ii}.\plabel{fix5:6}) Choose a compatible lift of $\theta\circ\pi^{-1}(\gamma(0))$ to a line embedded in the leaf space of $\widetilde{\F^s}$. Here, compatible means that the lift of $\theta\circ\pi^{-1}(\gamma(0))$ agrees with the lift of $\gamma$ at $\gamma(0)$. Similarly, choose a lift of $\theta\circ \pi^{-1}(\gamma(1))$ that agrees with the lift of $\gamma$ at $\gamma(1)$.\plabel{fix5:7} Then for $y_0\in \theta\circ\pi^{-1}(\gamma(0))$, we define $h_\gamma(y_0)$ to be the point in $\theta\circ\pi^{-1}(\gamma(1))$ whose lift to the leaf space of $\widetilde {\F^s}$ coincides with the lift of $y_0$, if such a point exists. Note that $h_\gamma$ is only defined on the subinterval of $\theta\circ\pi^{-1}(\gamma(0))$ whose lift intersects the lift of $\theta\circ\pi^{-1}(\gamma(1))$. See \cref{fig:holonomy} for a picture in $\Mpunc$. \plabel{fix:hgamma}

		\begin{figure}[ht]
			\centering
			\def\svgwidth{60mm}
\begingroup%
  \makeatletter%
  \providecommand\color[2][]{%
    \errmessage{(Inkscape) Color is used for the text in Inkscape, but the package 'color.sty' is not loaded}%
    \renewcommand\color[2][]{}%
  }%
  \providecommand\transparent[1]{%
    \errmessage{(Inkscape) Transparency is used (non-zero) for the text in Inkscape, but the package 'transparent.sty' is not loaded}%
    \renewcommand\transparent[1]{}%
  }%
  \providecommand\rotatebox[2]{#2}%
  \newcommand*\fsize{\dimexpr\f@size pt\relax}%
  \newcommand*\lineheight[1]{\fontsize{\fsize}{#1\fsize}\selectfont}%
  \ifx\svgwidth\undefined%
    \setlength{\unitlength}{226.77165222bp}%
    \ifx\svgscale\undefined%
      \relax%
    \else%
      \setlength{\unitlength}{\unitlength * \real{\svgscale}}%
    \fi%
  \else%
    \setlength{\unitlength}{\svgwidth}%
  \fi%
  \global\let\svgwidth\undefined%
  \global\let\svgscale\undefined%
  \makeatother%
  \begin{picture}(1,0.74251316)%
    \lineheight{1}%
    \setlength\tabcolsep{0pt}%
    \put(0,0){\includegraphics[width=\unitlength,page=1]{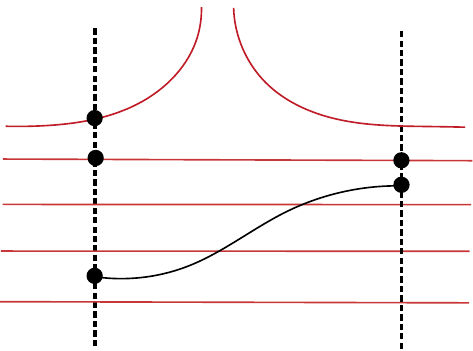}}%
    \put(0.8752609,0.33316896){\makebox(0,0)[lt]{\lineheight{1.25}\smash{\begin{tabular}[t]{l}$\gamma(1)$\end{tabular}}}}%
    \put(0.86679969,0.42228531){\makebox(0,0)[lt]{\lineheight{1.25}\smash{\begin{tabular}[t]{l}$h_\gamma(y_0)$\end{tabular}}}}%
    \put(0.85029115,0.69345689){\makebox(0,0)[t]{\lineheight{1.25}\smash{\begin{tabular}[t]{c}$\theta\circ\pi^{-1}(\gamma(1))$\end{tabular}}}}%
    \put(0.20104386,0.69832483){\makebox(0,0)[t]{\lineheight{1.25}\smash{\begin{tabular}[t]{c}$\theta\circ\pi^{-1}(\gamma(0))$\end{tabular}}}}%
    \put(0.18136179,0.51299089){\makebox(0,0)[rt]{\lineheight{1.25}\smash{\begin{tabular}[t]{r}$y_0'$\end{tabular}}}}%
    \put(0.17827479,0.42355703){\makebox(0,0)[rt]{\lineheight{1.25}\smash{\begin{tabular}[t]{r}$y_0$\end{tabular}}}}%
    \put(0.17759619,0.14559869){\makebox(0,0)[rt]{\lineheight{1.25}\smash{\begin{tabular}[t]{r}$\gamma(0)$\end{tabular}}}}%
  \end{picture}%
\endgroup%

			\caption{This is a picture in $\Mpunc$. The leaves of $\F^s$ are shown in red. The point $h_\gamma(y_0)$ lies on the same stable leaf as $y_0$. Moreover, there exists a path from $y_0$ to $h_\gamma(y_0)$ in a leaf of $\F^s$, and this path is homotopic to $\gamma$ through paths connecting $\theta\circ\pi^{-1}\gamma(0)$ to $\theta\circ\pi^{-1}(\gamma(1))$. These properties uniquely specify $h_\gamma(y_0)$. On the other hand, $h_\gamma$ is not defined at $y_0'$.}\label{fig:holonomy}
		\end{figure}

		Now we shall define a partial connection $\Jpar$ on the bundle $\pi$. We initially define $\Jpar$ on three types of short paths which generate a dense set of paths in $\Mpunc$. Whenever $h_\gamma$ exists, we define $\Jpar_\gamma$ so that 
		\begin{equation}\label{eqn:hol}
			\theta \circ \Jpar_\gamma = h_\gamma \circ \theta.
		\end{equation}
		In other words, the $\theta$-image of a parallel section tries to stay on the same leaf of $\F^s$. This is what happens in type 1 and type 2 curves below. However, $\Jpar$ will also be defined on some curves (type 3 curves below) along which holonomy does not exist.
		
		\begin{enumerate}[label=\textbf{Type~\arabic*}]
			\item Suppose $\gamma$ is a path in a generalized leaf $\mu$ of $\F^u$. In this case, we simply define $\Jpar$ using the holonomy of $\F^s$. In $\widetilde \mu \cong \R^2$, the leaves of $\widetilde{\F^s}\cap \widetilde{\mu}$ (ie the lifts of flow lines of the suspension flow of $\phi$) and the lifts to $\widetilde \mu$ of the $\theta$-images of fibers of $\pi$ form two transverse codimension-1 foliations. This pair of foliations is topologically conjugate to the horizontal and vertical foliations of $\R^2$. It follows that $h_\gamma$ is fully defined on $\theta\circ \pi^{-1}(\gamma(0))$. We may thus define $\Jpar_\gamma$ as in \cref{eqn:hol}. In other words, we are pushing $\theta$-images of fibers of $\pi$ along the suspension flow of $\phi$. \plabel{fix2:24}
			
			\item Suppose $R$ is a rectangle in $\Sigmapunc_p$ for some $p$. Let $\gamma$ be a path in $R$ from the east to the west side of $R$.\plabel{fix:27}\plabel{fix2:25}\plabel{fix2:28:1} In this case, $h_\gamma$ identifies the east and west sides of the rectangle by the obvious isometry, so we define $\Jpar$ as in \cref{eqn:hol}. (If an endpoint of $\gamma$ lies on a prong, then we make the same definition of $\Jpar_\gamma$ regardless of whether the endpoint is infinitesimally perturbed to the east or the west. We also allow the north or south side of the rectangle to contain a singularity.)
			
			\item \plabel{fix:41}Let $p$ be a point on a prong $P$ of $\F^u \cap \Sigmapunc_p$\plabel{fix5:13} and let $\gamma$ be the infinitesimal path from $p_{\xleft}$ to $p_{\xright}$. Let $q$ be the singularity terminating $P$. Let $P_{\xright}$ and $P_{\xleft}$ be the prongs at $q$ adjacent to $P$ so that $$\theta\circ\pi\inv(p_{\xleft})=P\cup \set{q^*} \cup P_{\xleft}$$ and $$\theta\circ\pi\inv(p_{\xright})=P\cup \set{q^{**}} \cup P_{\xright}.$$ Briefly conflating $\theta\circ\pi^{-1}(p_{\xleft})$ with $\pi^{-1}(p_{\xleft})$, let us define $$\Jpar_{\gamma}\colon P \cup \set{q^*} \cup P_{\xleft} \to P\cup \set{q^{**}} \cup P_{\xright}$$ separately in two subcases:
			\begin{enumerate}
				\item[\textbf{Type 3a}] If $p$ lies to the south of $q$, then declare that $\Jpar_\gamma$ sends $q^*$ to $q^{**}$, acts as the identity on $P$, and stretches $P_{\xleft}$ by a factor of $\lambda^{m_iq_i/p_i}$, where $(p_i;q_i)$ is the surgery slope at the relevant singularity and $m_i$ is the $\phi$-period of the singularity. Here, $\lambda$ is the stretch factor of $\phi$ and the relevant metric is the singular Euclidean metric induced on $\widetilde{\Sigmapunc_p}$ by the incomplete Solv metric $g$, as introduced at the beginning of \cref{subsec:partial}.\plabel{fix:30}\plabel{fix3:15}\plabel{fix5:9} It is instructive to look at \cref{fig:push} which illustrates how parallel sections over type 3a curves behave.
				\item[\textbf{Type 3b}] If $p$ lies to the north of $q$, then make the same definition except that we use the inverse stretch factor $\lambda^{-m_iq_i/p_i}$ instead.
			\end{enumerate}			
			Note that $h_\gamma$ is defined on $P$, but not on $P_{\xleft}$. Thus, we had some freedom to choose the dilation factors in type 3a and 3b. Our choice is designed to make \cref{prop:fillable} work.
		\end{enumerate}

		Now we will extend the definition of $\Jpar$ to arbitrary curves in $\Mpunc$. By composition of parallel transport maps, we may define $\Jpar$ for concatenations of type 1-3 curves. Given an arbitrary path $\gamma$ in $\Mpunc$, we may $C^0$ approximate $\gamma$ by a concatenation of type 1-3 curves. There may be many different ways to do this approximation  resulting in different parallel transport maps with slightly different domains of definition; for example, a type 2 curve may be written as a concatenation of a type 2, type 3, and a type 2 curve. We define $\Jpar_\gamma$ by patching together all of these parallel transport maps. More precisely, for $x\in\pi^{-1}(\gamma(0))$, we say that $\Jpar_\gamma(x)=y$ if there is a sequence of paths $\gamma_i$ from $\gamma(0)$ to $\gamma(1)$ such that
		\begin{itemize}\plabel{fix5:10}
			\item $\gamma_i$ is a concatenation of type 1-3 curves,
			\item the sequence of paths $\gamma_i$ converges to $\gamma$ in the $C^0$ topology as $i \to \infty$, and
			\item \xlabel{fix6:4}{$\Jpar_{\gamma_i}(x)=y$.}
		\end{itemize}
		\cref{prop:flat1} below guarantees that the various different $C^0$ approximations of $\gamma$ yield parallel transport maps that agree on the intersection of their domains of definition. So $\Jpar_\gamma$ is well defined.
		
		\begin{rmk}
			The definition of $\Jpar$ along Type 3 curves is motivated by \cref{cnstr:gluing}. The images in $L$ of $P_{\xleft}$ and $P_{\xright}$ are paths the form $[x,\infty)$ constructed in \cref{sec:gluing}. The failure of $\Jpar$ to be a connection is analogous to the failure of $L/\fequiv$ to be a line in \cref{cnstr:gluing}. \plabel{fix:32}\plabel{fix3:16}
		\end{rmk}

		\begin{defn}\label{defn:commutator}
			A \emph{commutator of type $i$ and type $j$} curves is a \xlabel{fix6:5}{contractible loop formed by concatenating four arcs which alternate between type $i$ and type $j$ curves}.\plabel{fix5:11}
		\end{defn}
		
		\begin{prop}\label{prop:flat1}
			Given any point $p\in \Mpunc$ and a compact subset $W$ of $\pi^{-1}(p)$, the monodromy of $\Jpar$ around sufficiently small contractible loops based at $p$ formed by concatenating type 1, 2, and 3 curves is defined on $W$ and equal to the identity. It follows that $\Jpar$ is flat.
		\end{prop}
		\begin{proof}
			Without loss of generality, assume $p\in \interior \theta(W)$. If $p$ does not lie on a prong, then one can find a \xlabel{fix6:6}{tall, skinny rectangle $R$ in $\Sigmapunc_p$} containing a neighbourhood of $\theta(W)$. Thicken this rectangle by $\eps$ in the direction transverse to $\Sigmapunc_p$. Now holonomy of transversals to $\F^s$ in $R\times (-\eps,\eps)$ exists along all curves in a neighbourhood $p$, so the required monodromies are trivial in this case.

			Suppose instead that $p=p_{\xleft}$ lies on a prong $P\subset \Sigmapunc_p$. This time, choose a tall, skinny rectangle $R_{\xleft}$ such that its east side contains $\theta(W)$. Choose another tall, skinny rectangle $R_{\xright}$ such that its west side contains $\theta(\Jpar_\gamma W)$, where $\gamma$ is the type 3 curve from $p_{\xleft}$ to $p_{\xright}$. Let $U$ be an $\varepsilon$-thickening of $R_{\xleft}\cup R_{\xright}$ in the direction transverse to $\Sigmapunc_p$. By construction, $p$ lies on the interior of $U$.

			We need to show that monodromy around a contractible loop in $U$ based at $p$ is trivial on $W$. The null-homotopy of such a loop may be decomposed into disks whose boundaries are commutators of type $i$ and type $j$ arcs. Moreover, we can arrange that the only commutators crossing $P\times(-\eps,\eps)$ are commutators of type 3 and type 1 arcs. This can be achieved by subdividing every type 2 arc crossing $P\times (-\eps, \eps)$ into a concatenation of a type 2 arc in $R_{\xleft} \times (-\eps,\eps)$, a type 3 arc crossing $P\times (-\eps, \eps)$, and a type 2 arc in $R_{\xright}\times(-\eps,\eps)$. This replacement only increases the domain of definition of the $\Jpar$. Any commutator contained in $R_{\xleft}\times(-\eps, \eps)$ or $R_{\xright}\times(-\eps, \eps)$ has trivial monodromy on $W$. In particular, this includes any commutator involving a type 2 arc. So it only remains to check that the monodromy around a commutator of a type 1 arc and a type 3 arc vanishes.\plabel{fix5:12} Parallel transport along a type 1 arc which moves a distance $\alpha$ in the direction transverse to $\Sigmapunc_p$ acts as a dilation by factor $\lambda^\alpha$ (relative to the induced Euclidean metric on fibers). On the other hand, parallel transport along a type 3 arc acts as a piecewise dilation. These two maps commute, as desired. Finally, it's easy to check that parallel transport along any path $\gamma\subset U$ starting at $p$ is defined on $W$.
		\end{proof}
		
		\plabel{fix3:17}
		Now we will prove one of the key properties that make our constructions work.
		\begin{prop}\label{prop:fillable}
			The monodromy of $\Jpar$ along a curve parallel to a filling slope \xlabel{fix6:7}{on a toroidal end} of $\Mpunc$ is trivial. To be precise, let $p$ be a point on an unstable prong incident to an end of $\Mpunc$. For any compact subspace $W \subset \pi^{-1}(p)$, the monodromy of $\Jpar$ around a small enough loop $\gamma$ based at $p$ and homotopic to the Dehn filling meridian is defined on $W$ and equal to the identity.
		\end{prop}
		\begin{proof}
			The boundary curve in $N(K_i)$ can be written as the suspension of a small perturbation of a point in $\Sigma\cap K_i$ under the map $\phi^{m_iq_i}$ and an arc in $\Sigmapunc$ travelling a clockwise angle of $2\pi p_i$ around the singularity. Parallel transport along the first arc stretches distances by a factor of $\lambda^{m_iq_i}$. Roughly speaking, the second arc contains $p_i$ subarcs of type 3b and $p_i$ subarcs which are (the inverse of) type 3a. Together, $\Jpar$-parallel transport along these arcs produces a dilation of a factor of $(\lambda^{-m_iq_i/p_i})^{p_i}$. Thus, the composite monodromy is dilation by a factor $\lambda^{m_iq_i}(\lambda^{-m_iq_i/p_i})^{p_i}=1$.
			
			\begin{figure}[ht]
				\centering
				\def\svgwidth{70mm}
				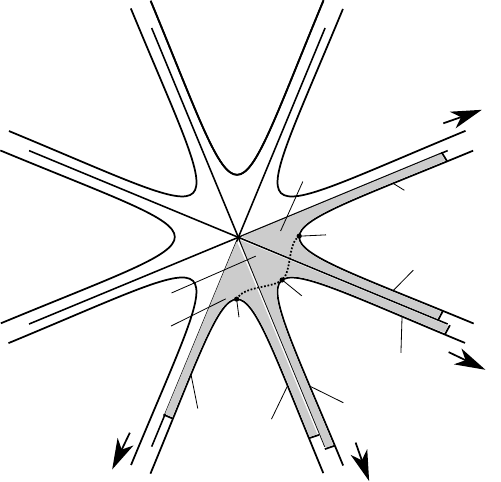
				\caption{}\label{fig:monodromy}
			\end{figure}

			In more detail, let us analyze parallel transport around a short arc travelling an angle of $2\pi$ around the singularity. \cref{fig:monodromy} shows the foliation $\F^{u}\cap \Sigmapunc$ near a singularity along with an arc $\gamma$ travelling an angle of $2\pi$ around the singularity. The arc $\gamma$ is chosen so that $\gamma(0),\gamma(1/2)$ and $\gamma(1)$ lie on prongs of the \emph{stable} foliation $\F^{s}\cap \Sigmapunc$ (not shown). 
			The $\theta$-images of half fibers of $\pi$ are labelled $P_1\dots P_6$, so that, for example, $\theta \circ \pi\inv(\gamma(0))=P_1\cup P_2$ and $\gamma(0)=P_1\cap P_2$. We will abuse notation and use $P_1$ to refer to both a half leaf $\F^u \cap \Sigmapunc$ and its $\theta$-preimage in $\pi^{-1}(\gamma(0))$, and similarly for $P_2\dots P_6$.
			
			Let $R_1$, $R_2$, and $R_3$ be rectangles bounded on the east and west sides by the $P_i$'s and the prongs of $\F^u \cap \Sigmapunc$  as shown shaded in \cref{fig:monodromy}. We might not be able to make the rectangles arbitrarily long in the north south direction due to hitting other singularities. However, if $\gamma(0)$, $\gamma(1/2)$, and $\gamma(1)$ hug very close to the singularity, these rectangles can be chosen to be as long in the north-south direction as desired.

			Restrict attention to points in the $P_i$'s such that the $\theta$-images of their parallel transports along $\gamma$ stay inside $R_1 \cup R_2 \cup R_3$. Call the set of such points $N$. For example, $N\cap (P_1\cup P_2)$ is an interval containing $\gamma(0)$.
			
			Now $\gamma|_{[0,1/2]}$ decomposes as a type 2 curve\plabel{fix2:28:2} in $R_1$, the inverse of a type 3a curve crossing the prong, and a type 2 curve in $R_2$. Therefore, $\Jpar_{\gamma|_{[0,1/2]}}$ maps $P_2\cap N$ to $P_3\cap N$ by an isometry and maps $P_1\cap N$ to $P_4\cap N$ by a stretch by a factor $\lambda^{-m_iq_i/p_i}$. (It is worth checking in \cref{fig:monodromy} that the prong is indeed crossed from east to west, so $\gamma$ follows the inverse of a type 3a curve and the corresponding dilation factor has negative exponent.)\plabel{fix2:28:4} Similarly, $\gamma|_{[1/2,1]}$ decomposes as a type 2 curve in $R_2$, a type 3b curve, and a type 2 curve in $R_3$. Thus, $\gamma\mid_{[1/2,1]}$ maps $P_4\cap N$ to $P_5\cap N$ by an isometry and $P_3\cap N$ to $P_6\cap N$ by a stretch by a factor $\lambda^{-m_iq_i/p_i}$.\plabel{fix2:28:5} The composite $\Jpar_\gamma$ maps $(P_1\cup P_2)\cap N$ to $(P_5\cup P_6)\cap N$ by a dilation of $\lambda^{-q_im_i/p_i}$. This justifies and makes precise the claim made in the first paragraph of the proof. 
			
			Finally, if we take $\gamma(0), \gamma(1)$, and $\gamma(1/2)$ to hug close to the singularity, then $R_1,R_2$ and $R_3$ may be made as long as needed, and in turn $N$ may be made as large as desired. By the accounting of dilation factors from the first paragraph, the monodromy of $\gamma$ is the identity on as large a subspace of $\pi^{-1}(p)$ as desired.

		\end{proof}
		
	\end{subsection}
	
	\begin{subsection}{Blowup time for parallel sections}\label{sec:blowup}
		Let $\gamma\colon [0,\infty)\to\Sigmapunc$ parameterize\plabel{fix3:18} an eastward ray in a leaf of $\Sigmapunc \cap \F^s$ by arclength.\plabel{fix2:41:1} (Here, we are conflating $\Sigmapunc$ with its lift to $Z$, and will continue to do so without comment. So, for example, $\gamma(0)$ might be of the form $p_{\xleft}$ or $p_{\xright}$ for some $p$ on a prong.)\plabel{fix2:29} Given a point $x\in \pi\inv(\gamma(0))$, recall from \cref{sec:connectionintro} that one may attempt to contruct a parallel section $s_x\colon [0,\infty) \to E_\pi$ with $s_x(0)=x$. However, $s_x$ typically blows up in finite time and can only be defined over $[0,t_{\max{}}(x))$ for some $t_{\max{}}(x) > 0$. In this section, we show that $t_{\max{}}$ is locally a homeomorphism from $\pi^{-1}(\gamma(0))$ to $\R$.

		\begin{rmk}\label{rmk:abuse}
			In what follows, we will abuse notation and treat $s_x$ as a real-valued function of $t$ whose value is the signed north-south distance in $\widetilde{\Sigmapunc}$ from $\gamma(t)$ to $\theta\circ s_x(t)$. Similarly, we will think of $x$ as a real number. 
		\end{rmk}\plabel{fix2:30}
		
		Recall that we defined $\lambda$ so that $\lambda > 1$. We also assumed that $p_i$ all have the same sign. Without loss of generality, we assume $p_i\geq 0$. Let $\alpha_i=\lambda^{m_iq_i/p_i}$. Note that $\alpha_i\geq 1$. Let $\alpha_{\max{}}=\max_i \lambda^{m_iq_i/p_i}$. If $\alpha_{\max{}}=1$, then $\MM$ is fibered and there is no blowup of sections. So we assume that $\alpha_{\max{}}>1$. Now we can define $\alpha_{min{}}=\min\set{\lambda^{m_i q_i/p_i} \mid q_i > 0}$.  We call a singularity \emph{magnifying} if the associated coefficient satisfies $\alpha_i>1$. A \emph{ragged rectangle} is the region in $\widetilde{\Sigmapunc}$ swept out by a interval in $\F^{s}\cap\Sigmapunc$ under the time $t$ northward flow along $\F^u\cap\Sigmapunc$ (ie the map that sends each point $t$ units due north). It looks like a rectangle with some vertical slits cut into the northern edge. 
		
		It is best to visualize all our constructions via the developing map $D\colon \widetilde{\Sigmapunc} \to \R^2$ which is locally an isometry.\plabel{fix2:31} We can arrange that $D(\widetilde{\F^s\cap \Sigmapunc})$ and $D(\widetilde{\F^u\cap \Sigmapunc})$ are the horizontal, resp. vertical foliations of $\R^2$. Thus, the $\theta$-images of fibers of $\pi$ also correspond to vertical lines in this picture. Ragged rectangles project via $D$ to honest rectangles. We can also arrange that the image of $\gamma$ is the positive $x$ axis.
		
		Given $x \in \pi\inv(\gamma(0))$, the section $s_x(t)$ may be constructed geometrically. If $x>0$, push the vertical segment $[\gamma(0),\theta(x)]$ eastward along $\gamma$. The endpoint sweeps out the $\theta$ image of the section $s_x(t)$. See \cref{fig:push}. Whenever the segment hits a singularity in orbit $i$ at distance $d$ from $\gamma(t)$ with $0 < d < s_x(t)$, the segment continues on the other side with length $d + \alpha_i(s_x(t)-d)$. If $x < 0$, then the same procedure works except that the new length of the segment is $d+(1/\alpha_i)(s_x(t)-d)$. 
		
		\begin{figure}[ht]
			\centering
			\def\svgwidth{90mm}
			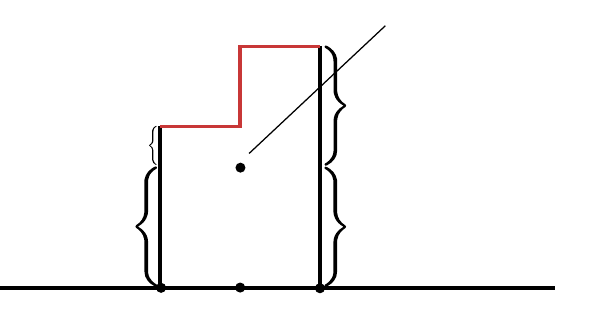
			\caption{Constructing $s_x(t)$ by pushing a fiber of $\pi$ past a singularity, as seen in the image of the developing map.}\label{fig:push}
		\end{figure}

		When $x<0$, the section $s_x(t)$ stays bounded for all time because $\gamma$ can be expressed as a composite of type 2 and type 3b curves, neither of which increases $|s_x(t)|$. On the other hand, when $x > 0$, $s_x(t)$ grows with $t$ and possibly blows up. \cref{lem:ergodic} below guarantees that we can expect many singularities to be encountered during this procedure. 
		
		\begin{figure}[ht]
			\centering
			\def\svgwidth{120mm}
\begingroup%
  \makeatletter%
  \providecommand\color[2][]{%
    \errmessage{(Inkscape) Color is used for the text in Inkscape, but the package 'color.sty' is not loaded}%
    \renewcommand\color[2][]{}%
  }%
  \providecommand\transparent[1]{%
    \errmessage{(Inkscape) Transparency is used (non-zero) for the text in Inkscape, but the package 'transparent.sty' is not loaded}%
    \renewcommand\transparent[1]{}%
  }%
  \providecommand\rotatebox[2]{#2}%
  \newcommand*\fsize{\dimexpr\f@size pt\relax}%
  \newcommand*\lineheight[1]{\fontsize{\fsize}{#1\fsize}\selectfont}%
  \ifx\svgwidth\undefined%
    \setlength{\unitlength}{1017.96948242bp}%
    \ifx\svgscale\undefined%
      \relax%
    \else%
      \setlength{\unitlength}{\unitlength * \real{\svgscale}}%
    \fi%
  \else%
    \setlength{\unitlength}{\svgwidth}%
  \fi%
  \global\let\svgwidth\undefined%
  \global\let\svgscale\undefined%
  \makeatother%
  \begin{picture}(1,0.54337835)%
    \lineheight{1}%
    \setlength\tabcolsep{0pt}%
    \put(0,0){\includegraphics[width=\unitlength,page=1]{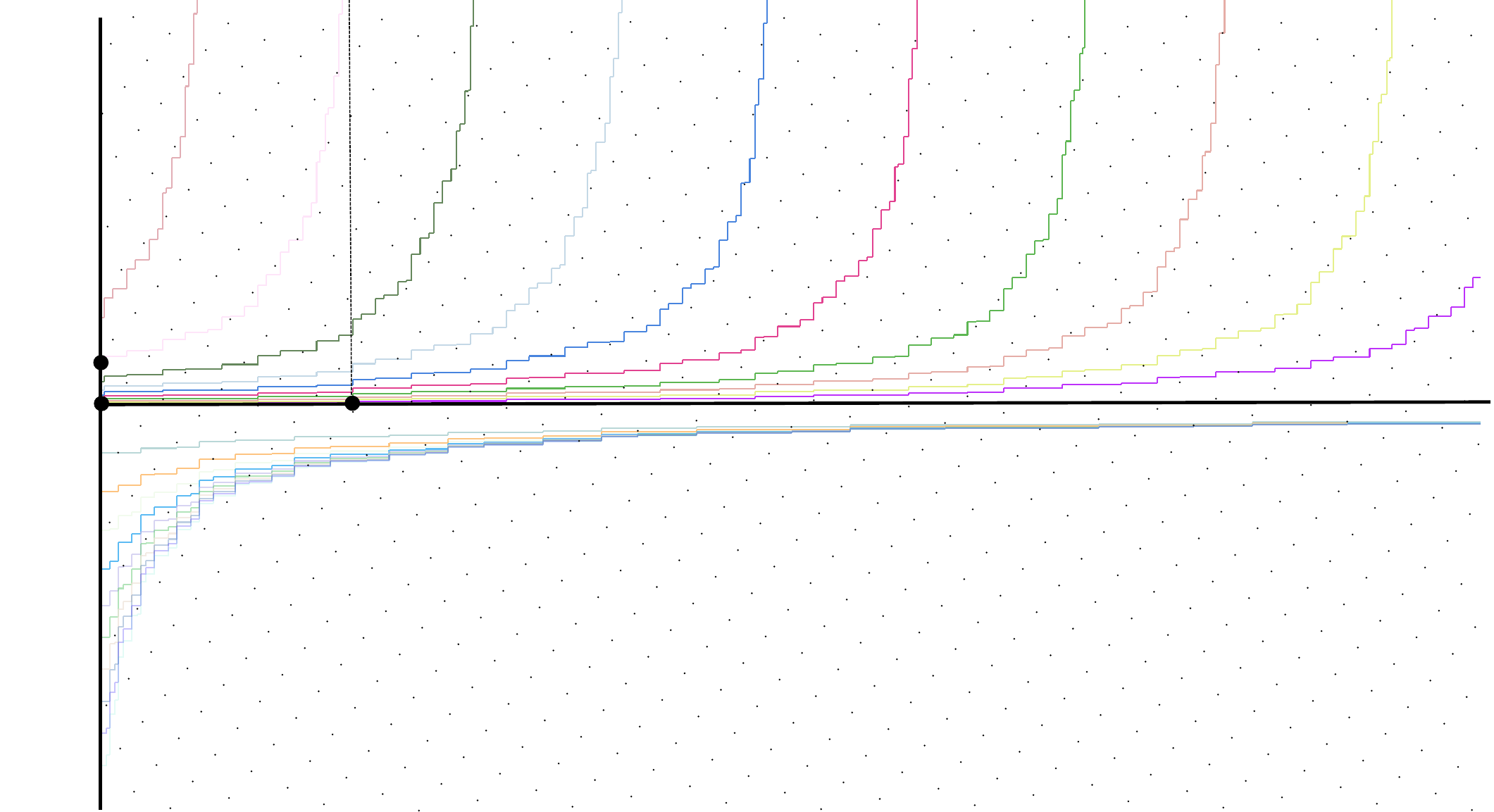}}%
    \put(0.05957868,0.29258818){\color[rgb]{0,0,0}\makebox(0,0)[rt]{\lineheight{1.25}\smash{\begin{tabular}[t]{r}$\theta(x)$\end{tabular}}}}%
    \put(0.05966596,0.26409412){\color[rgb]{0,0,0}\makebox(0,0)[rt]{\lineheight{1.25}\smash{\begin{tabular}[t]{r}$\gamma(0)$\end{tabular}}}}%
    \put(0.05294292,0.35994975){\color[rgb]{0,0,0}\rotatebox{90}{\makebox(0,0)[lt]{\lineheight{1.25}\smash{\begin{tabular}[t]{l}$\theta\circ \pi^{-1}(\gamma(0))$\end{tabular}}}}}%
    \put(0.22569907,0.24055668){\color[rgb]{0,0,0}\makebox(0,0)[lt]{\lineheight{1.25}\smash{\begin{tabular}[t]{l}$\gamma(t_{\rm max}(x))$\end{tabular}}}}%
  \end{picture}%
\endgroup%

			\caption{In this example, $\Sigmapunc=T^2$, $\phi$ is the monodromy of the figure eight knot, and the surgery coefficient is $(p_1;q_1)=(5,1)$. This picture shows $D(\widetilde{\Sigmapunc})\cong \R^2$ with the 2-pronged singularities of $\phi$ drawn as dots. Objects in $\Sigma$ have been lifted to $\widetilde{\Sigmapunc}$, and then projected to $\R^2$ via $D$; for convenience in labelling the picture, we identify objects with their images under this procedure. Sections $s_x(t)$ are shown in colours for various choices of $x$. These sections have been projected to $\Sigmapunc$ by $\theta$, then lifted to $\widetilde \Sigmapunc$, and finally projected to $D(\widetilde \Sigmapunc)$. Whenever $x\in\pi\inv(\gamma(0))$ satisfies $x<0$, the section $s_x(t)$ exists for all time. When $x>0$, $s_x(t)$ blows up in finite time.}\label{fig:blowup}
		\end{figure}
		\begin{lem}\label{lem:ergodic}
			There exists a positive constant $\kappa$ such that for every $\eps> 0$, there exists a large enough $A_\eps$ so that every ragged rectangle of area $A>A_\eps$ has number of magnifying singularities in the range $[(\kappa-\eps)A, (\kappa+\eps)A]$.
		\end{lem}
		\begin{proof}
			First, we may use the action of $\phi$ to turn any ragged rectangle of area $A$ into a new ragged rectangle of width $O(1)$ and height $O(A)$ having the same number of magnifying singularities. Now the lemma follows from the ergodicity of the translation flow on flat surfaces. Masur's criterion states that the translation flow on a flat surface is ergodic whenever the corresponding flow on Teichm\"{u}ller space given by multiplying the metric by $\begin{pmatrix} e^{t} & 0 \\ 0 & e^{-t} \end{pmatrix}$ stays in some compact set \cite[Sec. 3.7]{zorich_flat_2006}. A pseudo-Anosov map corresponds with a closed orbit under this flow. Therefore, Masur's criterion is fulfilled.
		\end{proof}
		
		We now define $A_*$ to be a constant large enough that every ragged rectangle of area $A_*$ contains a magnifying singularity. We also remind the reader that in what follows, we are using the identifications introduced in \cref{rmk:abuse}.\plabel{fix3:20}
		
		\begin{lem}\label{lem:fastexplode}\plabel{fix:42}
			There exists a constant $C$ independent of $\gamma$ such that whenever $x > S > 0$, $s_x(t)$ blows up in time $t_{\max{}}(x) < C/S$.
		\end{lem}
		\begin{proof}
			By \cref{lem:ergodic}, a prong from a magnifying singularity intersects the image of $\gamma$ in $\Sigmapunc$ at north-south distance $<\frac S 2$ in time $t<A_*\cdot \frac 2 S$. Therefore,
			
			\begin{align}
				s_x\paren{\frac {2A_*} S} &> \frac S 2 + \alpha_{\min{}} \frac S 2 \\
				&=\frac{1+\alpha_{\min{}}}{2}S
			\end{align}\plabel{fix2:33}
			Repeating the argument to find more nearby singularities, we find that
			\begin{equation}
				s_x\paren{\frac {2A_*}{S}\sum_{i=0}^{n-1} \paren{\frac {1+\alpha_{\min{}}} 2}^{-i}} > \paren{\frac {1+\alpha_{\min{}}} 2}^{n}S
			\end{equation}\plabel{fix2:34}
			
			as long as the left side is defined. Taking $n\to\infty$, we find that
			
			\begin{equation}
				t_{\max{}}(x) < \frac {2 A_*}{S} \paren{1-\frac 2 {1+\alpha_{\min{}}}}\inv.
			\end{equation}
			The right side is of the form $C/S$ as desired.
		\end{proof}

		\begin{lem}\label{lem:slowexplode}
			There exists a constant $c$ independent of $\gamma$ such that whenever $0\leq x < S$, $s_x(t)$ exists for time $c/S$.
		\end{lem}
		\begin{proof}
			By \cref{lem:ergodic}, we can choose a constant $A_{\kappa}$ so that every ragged rectangle of area $A>A_{\kappa}$ has fewer than $2\kappa A$ singularities. Choose $c={A_{\kappa}}/\paren{(\alpha_{\max{}})^{2\kappa A_{\kappa}}}$. Choose any $S>0$ and any $0\leq x < S$.\plabel{fix2:35}
			
			Consider a ragged rectangle $R$ whose southwest corner lies at $\gamma(0)$ and which has an area $A_\kappa$, with height $(\alpha_{\max{}})^{2\kappa A_{\kappa}} S$ and width $c/S={A_{\kappa}}/\paren{(\alpha_{\max{}})^{2\kappa A_{\kappa}}S}$. This ragged rectangle has fewer than $2\kappa A_\kappa$ singularities. Since $|s_x(t)|$ grows by at most $\alpha_{\max}$ each time it encounters a singularity, $\theta\circ s_x(t)$ does not grow enough times to exit through the top of $R$. Therefore, $s_x(t)$ exists for time at least the width of the ragged rectangle, $c/S$.
		\end{proof}
		\begin{prop}\label{prop:surjective}
			$t_{\max}\colon \{x\in \pi\inv(\gamma(0)) \mid x>0\} \to (0,\infty)$ is surjective.
		\end{prop}\plabel{fix2:37}
		\begin{proof}
			We first show that the image of $\tmax$ is dense. For any $t_0$ and any $\eps>0$, \cref{lem:fastexplode} guarantees that if $x\in \pi\inv(\gamma(t_0))$ is large enough, the parallel section over $\gamma$ passing through $x$ will blow up before $t=t_0+\eps$. A section can always be extended backwards in time because as one travels west, the section is always decreasing and positive.\plabel{fix2:36} See \cref{fig:blowup}. Thus, we have found a section defined at $t=0$ and which blows up between time $t_0$ and time $t_0+\eps$. Our choices of $t_0$ and $\eps$ were arbitrary, so $\tmax$ has dense image.
			
			Since $\tmax$ is non-decreasing and has dense image, it is surjective.
		\end{proof}

		\begin{prop}\label{prop:injective}
			$t_{\max}\colon \{x\in \pi\inv(\gamma(0))\mid x>0\} \to (0,\infty)$ is injective.
		\end{prop}
		\begin{proof}
			Choose $x,y\in \pi\inv(\gamma(0))$. Assume $0<x<y$. We want to show that $s_y$ explodes strictly before $s_x$.
			\begin{enumerate}[label=\textbf{Case~\arabic*}]
				\item $s_y(t)/s_x(t)$ is unbounded. Choose $t$ so that $s_y(t)/s_x(t)>C/c$, where $C$ and $c$ are the constants from \cref{lem:fastexplode} and \cref{lem:slowexplode}. By \cref{lem:fastexplode} with $S=s_y(t)$, $s_y$ blows up before time $t+C/s_y(t)$. By \cref{lem:slowexplode} with $S=s_y(t)c/C$, $s_x$ exists at time $t+C/s_y(t)$. Therefore, $s_y$ explodes strictly before $s_x$.\plabel{fix2:38}
				\item $s_y(t)/s_x(t)$ stays bounded. Let $\set{t_j}$ be the sequence of times at which $s_x$ grows, and let $b_j$ be the height of the singularity encountered at time $t_j$. Note that $0 < b_j < s_x(t_j)$. Let $t_j-\eps$ be a time just before crossing the singularity and $t_j+\eps$ a time just after. This helps us avoid talking about $s_x(t_j)$ which is problematic since $s_x$ is discontinuous at $t_j$. We have
				\begin{align*}
					\frac{s_y(t_j+\eps)-s_x(t_j+\eps)}{s_x(t_j+\eps)} &= \frac{\alpha_{t_j} \cdot (s_y(t_j-\eps)-s_x(t_j-\eps))} {\alpha_{t_j}\cdot(s_x(t_j-\eps)-b_j)+b_j}\\
					&=\frac{s_y(t_j-\eps)-s_x(t_j-\eps)} {s_x(t_j-\eps)-b_j(1-1/\alpha_{t_j}) }\\
					&=\frac 1 {1-(1-1/\alpha_{t_j})\cdot b_j/s_x(t_j-\eps)}\cdot\frac{s_y(t_j-\eps)-s_x(t_j-\eps)} {s_x(t_j-\eps)}\\
				\end{align*}\plabel{fix5:16}
				\plabel{fix3:22} where $\alpha_{t_j}$ is the magnification factor at the singularity encountered at time $t_j$. Since $\alpha_{t_j}>1$, we have $0 <  1- 1 / \alpha_{t_j} < 1$.  Since $\alpha_{t_j}$ takes on at most finitely many values, $1-1/\alpha_{t_j}$ is bounded away from 0. If $b_j/s_x(t_j-\eps)$ is also bounded away from 0\plabel{fix3:23}, then $(s_y(t)-s_x(t))/s_x(t)$ grows by a factor bounded away from 1. Since we assumed that $s_y(t)/s_x(t)$ remains bounded, we must have $\lim_{j\to\infty} b_j/s_x(t_j-\eps)=0$.\plabel{fix2:39}
				
				Let $N\in \N$ and $\eps > 0$ be constants to be determined. Choose $n$ large enough that $b_i/s_x(t_i) < \eps$ for all $i>n$. By \cref{lem:fastexplode}, there exists $m>n$ such that 
				\begin{equation}\label{eqn:compare}
					(\alpha_{\max{}})^{N-1} s_x(t_n) < s_x(t_m) < (\alpha_{\max{}})^{N+1} s_x(t_n)
				\end{equation}
				and
				\begin{equation}
					t_m - t_n < \frac C {s_x(t_n)}
				\end{equation}
				
				Now consider the ragged rectangle $R$ with base $\gamma([t_n,t_m])$ and height $\eps s_x(t_m)$. We have that
				\begin{align}
					\Area(R) &= \eps s_x(t_m)(t_m-t_n)\\
					&< \eps s_x(t_n) \frac C {s_x(t_n)} (\alpha_{\max{}})^{N+1}\\
					&= C\eps (\alpha_{\max{}})^{N+1}.
				\end{align}
				
				Since $b_j/s_x(t_j) < \eps$ for every $j\in [n,m]$, the ragged rectangle contains every singularity which the section $s_x$ encounters between time $t_n$ and $t_m$. Taking into account \cref{eqn:compare} and the fact that $s_x$ can grow by a factor of at most $\alpha_{\max{}}$ each time it encounters a singularity, there must be at least $N$ singularities in $R$. Taking $N$ large enough and then $\eps$ small enough, we find a ragged rectangle of arbitrarily small area containing arbitrarily many singularities. Such a rectangle may be extended northward to a ragged rectangle of any desired larger area with at least as many singularities.\plabel{fix2:42} This contradicts \cref{lem:ergodic}.\qedhere
			\end{enumerate}
		\end{proof}
		
	\end{subsection}
	
	\begin{subsection}{The full connection \titletex{$\J$}}\label{sec:fullconnection}
		The uniqueness statements for blowup proven in the previous section suggest that we add a point at infinity to each fiber of $\pi$, obtaining an new bundle $S^1 \to E_\Pi \xrightarrow{\Pi} \Mpunc$. Via the embedding of fibers of $\pi$ into fibers of $\Pi$, $\Jpar$ is naturally a partial connection on $\Pi$ as well. We now define an honest connection $\J$ on $\Pi$ which extends $\Jpar$ (ie for all $\gamma\subset \Mpunc$, we have $\J_\gamma=\Jpar_\gamma$ on the domain of $\Jpar_\gamma$ in $\Pi^{-1}(\gamma(0))$).\plabel{fix3:24}
		
		On a type 1 curve $\gamma$, the domain of $\Jpar_\gamma$ is all of $\pi^{-1}(\gamma(0))$. We define $\J$ in the only sensible way which extends $\Jpar$, asking that $\J_\gamma$ sends the point at infinity in $\Pi^{-1}(\gamma(0))$ to the point at infinity in $\Pi^{-1}(\gamma(1))$.
				
		It remains to define $\J$ on curves oriented east-west, ie those in leaves of $ \Sigmapunc\cap \F^s$\plabel{fix2:43} which were considered in \cref{sec:blowup}. Such a curve may be represented as a (possibly infinite) concatenation of type 2 and type 3 curves.\plabel{fix3:25} Given $\gamma:[0,t]\to \Mpunc$ parameterizing an eastward interval in a leaf of $\Sigmapunc \cap \F^s$ by arclength\plabel{fix2:41:2} and a point $x \in \Pi\inv(\gamma(0))$, let $s_0$ be the parallel section over $\gamma$ traced out by parallel transport of $x$ along $\gamma$ using $\Jpar$. Suppose $s_0$ blows up towards $+\infty$ at time $\tmax(x)$. (When $x=\infty$, we say that $\tmax(x)=0$, and when $x<0$ we have $\tmax(x)=\infty$.) By \cref{prop:surjective} and \cref{prop:injective}, there exists a unique parallel section $s_1$ defined on $(\tmax(x),\infty)$\plabel{fix2:44} that blows up towards $-\infty$ as $t$ approaches $\tmax(x)$ from the right. We now define $\J$ in cases:
		\begin{equation}\label{eqn:jdef}
			\J_{\gamma}(x)=\begin{cases}
				s_0(t) &\text{if } t<\tmax(x)\\
				s_1(t) &\text{if } t>\tmax(x)\\
				\infty &\mbox{if } t=\tmax(x)
			\end{cases}
		\end{equation}
		See \cref{fig:sec1}.
			
			\begin{figure}[ht]
				\centering
				\def\svgwidth{0.9\linewidth}
\begingroup%
  \makeatletter%
  \providecommand\color[2][]{%
    \errmessage{(Inkscape) Color is used for the text in Inkscape, but the package 'color.sty' is not loaded}%
    \renewcommand\color[2][]{}%
  }%
  \providecommand\transparent[1]{%
    \errmessage{(Inkscape) Transparency is used (non-zero) for the text in Inkscape, but the package 'transparent.sty' is not loaded}%
    \renewcommand\transparent[1]{}%
  }%
  \providecommand\rotatebox[2]{#2}%
  \newcommand*\fsize{\dimexpr\f@size pt\relax}%
  \newcommand*\lineheight[1]{\fontsize{\fsize}{#1\fsize}\selectfont}%
  \ifx\svgwidth\undefined%
    \setlength{\unitlength}{508.91363525bp}%
    \ifx\svgscale\undefined%
      \relax%
    \else%
      \setlength{\unitlength}{\unitlength * \real{\svgscale}}%
    \fi%
  \else%
    \setlength{\unitlength}{\svgwidth}%
  \fi%
  \global\let\svgwidth\undefined%
  \global\let\svgscale\undefined%
  \makeatother%
  \begin{picture}(1,0.60003729)%
    \lineheight{1}%
    \setlength\tabcolsep{0pt}%
    \put(0,0){\includegraphics[width=\unitlength,page=1]{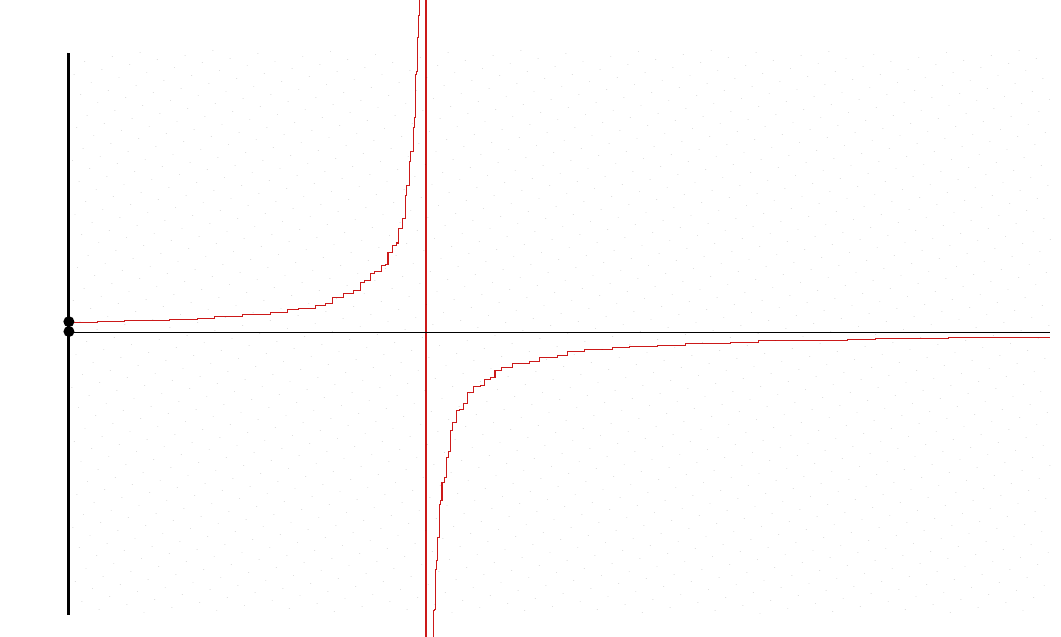}}%
    \put(0.37932393,0.4297229){\makebox(0,0)[rt]{\lineheight{1.25}\smash{\begin{tabular}[t]{r}$\theta\circ s_x(t)$\end{tabular}}}}%
    \put(0.05677032,0.2997535){\makebox(0,0)[rt]{\lineheight{1.25}\smash{\begin{tabular}[t]{r}$\theta(x)$\end{tabular}}}}%
    \put(0.0573098,0.27121164){\makebox(0,0)[rt]{\lineheight{1.25}\smash{\begin{tabular}[t]{r}$\gamma(0)$\end{tabular}}}}%
    \put(0.04979247,0.34668401){\rotatebox{90}{\makebox(0,0)[lt]{\lineheight{1.25}\smash{\begin{tabular}[t]{l}$\theta\circ\pi^{-1}(\gamma(0))$\end{tabular}}}}}%
  \end{picture}%
\endgroup%

				\caption{We continue the example from \cref{fig:blowup}. As before, we are drawing all objects lifted to $\widetilde \Sigmapunc$ and then projected via $D$ to $\R^2$. The figure shows the $\theta$-image of a parallel section for $\J$ over a curve $\gamma$ oriented east-west. The curve $\gamma$ is drawn as a horizontal black line in the figure. For any $t$, $\theta\circ \pi^{-1}(\gamma(t))$ would be a vertical line in the picture. The section is piecewise constant. It starts out equal to $s_0$, wraps around $\infty$ at $t=\tmax(x)$, and continues as $s_1$.}\label{fig:sec1}
			\end{figure}
			
			\begin{figure}[ht]
				\centering
				\def\svgwidth{0.9\linewidth}
\begingroup%
  \makeatletter%
  \providecommand\color[2][]{%
    \errmessage{(Inkscape) Color is used for the text in Inkscape, but the package 'color.sty' is not loaded}%
    \renewcommand\color[2][]{}%
  }%
  \providecommand\transparent[1]{%
    \errmessage{(Inkscape) Transparency is used (non-zero) for the text in Inkscape, but the package 'transparent.sty' is not loaded}%
    \renewcommand\transparent[1]{}%
  }%
  \providecommand\rotatebox[2]{#2}%
  \newcommand*\fsize{\dimexpr\f@size pt\relax}%
  \newcommand*\lineheight[1]{\fontsize{\fsize}{#1\fsize}\selectfont}%
  \ifx\svgwidth\undefined%
    \setlength{\unitlength}{509.62890625bp}%
    \ifx\svgscale\undefined%
      \relax%
    \else%
      \setlength{\unitlength}{\unitlength * \real{\svgscale}}%
    \fi%
  \else%
    \setlength{\unitlength}{\svgwidth}%
  \fi%
  \global\let\svgwidth\undefined%
  \global\let\svgscale\undefined%
  \makeatother%
  \begin{picture}(1,0.61478935)%
    \lineheight{1}%
    \setlength\tabcolsep{0pt}%
    \put(0,0){\includegraphics[width=\unitlength,page=1]{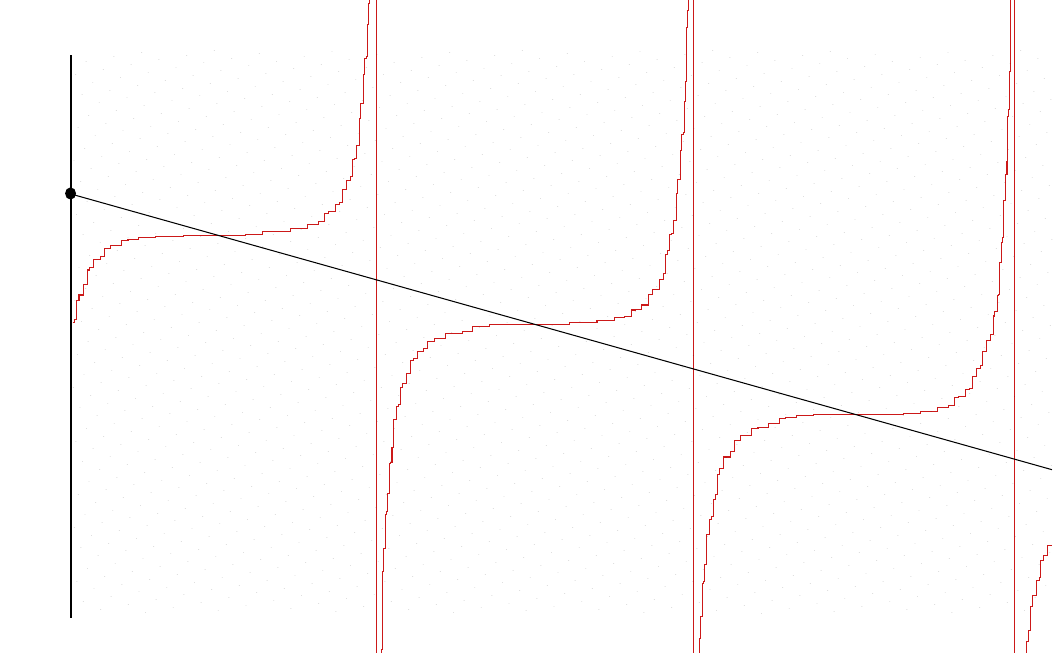}}%
    \put(0.33178436,0.47289095){\makebox(0,0)[rt]{\lineheight{1.25}\smash{\begin{tabular}[t]{r}$\theta\circ s_x(t)$\end{tabular}}}}%
    \put(0.05471157,0.42781415){\makebox(0,0)[rt]{\lineheight{1.25}\smash{\begin{tabular}[t]{r}$\gamma(0)$\end{tabular}}}}%
    \put(0.05718708,0.30563238){\makebox(0,0)[rt]{\lineheight{1.25}\smash{\begin{tabular}[t]{r}$\theta(x)$\end{tabular}}}}%
    \put(0.0517083,0.03464049){\rotatebox{90}{\makebox(0,0)[lt]{\lineheight{1.25}\smash{\begin{tabular}[t]{l}$\theta\circ\pi^{-1}(\gamma(0)$\end{tabular}}}}}%
    \put(0,0){\includegraphics[width=\unitlength,page=2]{fullsection2.pdf}}%
  \end{picture}%
\endgroup%

				\caption{The $\theta$-image of a parallel section for $\J$ over a curve $\gamma$ of negative slope relative to the stable foliation on $\Sigmapunc$. This picture is drawn with the same conventions as in \cref{fig:sec1}.  The section $s_x$ intersects the $\infty$-section infinitely many times. To see this, first homotope $\gamma$ in $\Sigmapunc$ to a curve which is piecewise southward or eastward with each straight segment having length at least $\eps$ for some $\eps>0$. This does not change $\theta\circ s_x$ since $\J$ is flat. Since $\gamma$ had negative slope and $\theta \circ s_x$ is non-decreasing, the north-south distance between $\theta\circ s_x(t)$ and $\gamma(t)$ increases to $C/\eps$ in finite time, where $C$ is the constant from \cref{lem:fastexplode}. By that lemma, $s_x$ blows up and crosses the $\infty$-section within the next eastward segment of $\gamma$. This process repeats indefinitely and produces infinitely many intersections between $s_x$ and the $\infty$-section.} \label{fig:sec2}
			\end{figure}		
		
		\begin{prop}\label{prop:flat2}
			$\J$ is flat.
		\end{prop}
		\begin{proof}
			Consider three types of short arcs in $\Mpunc$: eastward arcs, northward arcs, and transverse arcs (ie arcs contained in flowlines of the suspension flow of $\phi$). Every curve in $\Mpunc$ can be $C^0$ approximated by such curves. As in \cref{prop:flat1}, it is enough to show that the monodromy of $\J$ around all commutators is trivial. 
			
			First, let's check that the monodromy around a commutator of eastward and northward arcs vanishes (ie the monodromy around the boundary of a rectangle is trivial). Let $\gamma_1$ and $\gamma_2$ be the north and south sides of the rectangle respectively. The $\theta$-image of any $\Jpar$-parallel section over $\gamma_1$ is also the $\theta$-image of a $\Jpar$-parallel section over $\gamma_2$; this follows from flatness of $\Jpar$ combined with the fact that the $\theta$-image of a $\Jpar$-parallel section along any northward arc is a single point. A $\J$-parallel section is formed by patching together $\Jpar$ parallel sections at their asymptotes according to \cref{eqn:jdef}. Therefore, the $\theta$-images of $\J$-parallel sections over $\gamma_1$ are also $\theta$-images of $\J$-parallel sections over $\gamma_2$. It follows that the desired monodromy is the identity.\plabel{fix2:46}

			Second, the commutator of northward and transverse arcs vanishes because both are type 1 curves. 

			Finally, consider the commutator of a transverse arc of length $\eps$ and an eastward arc $\gamma$. Observe that the entire construction of $\J_\gamma$ for an eastward arc $\gamma$ is equivariant with respect to dilations in the north-south or east-west directions, and in particular, a dilation by $\lambda^\eps$ in the north-south direction and $\lambda^{-\eps}$ in the east-west direction where $\eps$. This is equivalent to the vanishing of the desired monodromy.
		\end{proof}
		We can now rectify the lack of a good topology on $E_\Pi$. A flat connection gives a local product structure on $E_\Pi$, so we endow $E_\Pi$ with the corresponding local product topology.
		
		$\Pi$ has a special section, called the \emph{$\infty$-section} and denoted $r_\infty$, whose value at any point is the point at infinity in the corresponding fiber. There is another section, called the \emph{zero-section} and denoted $r_0$, defined by the property $$\theta(r_0(p))=p$$ for all $p\in \Mpunc$. Note that neither of these sections is flat with respect to $\J$. Nevertheless, these sections are continuous.
		\begin{lem}
			The $\infty$-section and the zero-section are both continuous sections of $\Pi$.
		\end{lem}
		\begin{proof}
			\cref{lem:slowexplode} gives a quantitative bound on how fast flat sections can explode to $\infty$ in the east and west directions. In the north-south and transverse directions, flat sections do not blow up. It follows that the graph of the $\infty$-section is closed and $r_\infty$ is continuous.

			Now let's check continuity of the zero-section near a point $p\in \Mpunc$. Take a neighbourhood $U$ of a point $p\in \Mpunc$ which $\F^s$ and $\F^u$ foliate as products. The set $$U'=\set{y \in E_\pi \mid \theta(y)\in U, \pi(y)\in U}$$ is a neighbourhood of $r_0(p)$. Since holonomy of $\F^s$ exists for all transverse arcs in $U$, the topology on $U'$ is just the standard product topology. The graph of the zero section is clearly closed in this topology.
		\end{proof}

		\begin{lem}\label{lem:zeroflat}
			The zero-section is flat over any leaf of $\F^s$.
		\end{lem}
		\begin{proof}
			This follows immediately from the way we defined $\J$ to agree with the holonomy of $\F^s$ in \cref{eqn:hol}. If $\lambda$ is a leaf of $\F^s$ and $\gamma$ is a path in $\lambda$, then we have
			\begin{equation*}
				h_\gamma(\gamma(0)) = \gamma(1)
			\end{equation*}
			and
			\begin{align*}	
			\J_\gamma(r_0(\gamma(0))) &= \theta^{-1} \circ h_\gamma \circ \theta(r_0(\gamma(0)))\\
			&= \theta^{-1} \circ h_\gamma(\gamma(0))\\
			&= r_0(\gamma(1)).
			\end{align*}
		\end{proof}

		\begin{prop}\label{prop:fillable2}
			The monodromy of $\J$ around a curve parallel to the Dehn filling slope on a toroidal end of $\Mpunc$ is trivial. Moreover, a parallel section over such a curve has zero intersection number with the $\infty$-section.
		\end{prop}
		\begin{proof}
			Since $\J$ is a flat connection, \xlabel{fix6:10}{the conjugacy class of its monodromy around a closed curve is invariant under free homotopies of the closed curve.} Thus, it suffices to show that the desired monodromy around some freely homotopic curve is trivial. Fix a point $p$ on an unstable prong incident to the toroidal end. Consider a loop based at $p$ freely homotopic to our curve. \cref{prop:fillable} tells us that the monodromy of $\Jpar$ around this loop is the identity, and moreover that the domain of $\Jpar$ can be made as large as we like by choosing the loop to hug very close to the toroidal end. Since $\J$ agrees with $\Jpar$ when they are both defined, we conclude that the monodromy is the identity for all points in the fiber not equal to $\infty$. Since $\J$ is a bijection from $\Pi^{-1}(p)$ to itself, it must be equal to the identity at $\infty$ as well.\plabel{fix5:17}
			
			For the second statement, it suffices to check that just one parallel section of $\J$ has zero intersection number with the $\infty$-section. Simply take any parallel section of $\Jpar$ along the meridian; as noted in the previous paragraph, such a section has no intersections with the $\infty$-section.\plabel{fix2:49}
		\end{proof}
		\cref{prop:fillable2} guarantees that $(\Pi,\J)$ extends to a bundle over $\MM$ with a flat connection. Moreover, the $\infty$-section extends to a section over $\MM$. Thus, the Euler class of $\Pi$ vanishes and the $S^1$-bundle with flat connection $(\Pi, \J)$ unrolls to an $\R$-bundle with flat connection which we call $(\widehat \Pi, \widehat \J)$. We choose this unrolling so that the $\infty$-section lifts to a section of $\widehat \Pi$. More detail on this construction is provided in \cref{rmk:fiberwiselifts}.

		\begin{rmk}\label{rmk:fiberwiselifts}
			Suppose $X$ is a space with an $S^1$ bundle $S^1\to E \xrightarrow{\Pi} X$ and a flat connection $\H$. A \emph{fiberwise cover} of $\Pi$ is an bundle $\R \to \widehat E \xrightarrow{\widehat \Pi} X$ along with projection map $\beta$ making the following diagram commute.
			\begin{equation*}
				\begin{tikzcd}
					\R \arrow[r,""] \arrow["\mod 1", d, swap] & \widehat E \arrow[d,"\beta"] \arrow[dr,"\widehat{\Pi}"] &  \\
					S^1 \arrow[r,""] & E \arrow[r,"{\Pi}"] & X \\
				\end{tikzcd}
			\end{equation*}
			Fiberwise covers of $\Pi$ are classified by elements of $H^1(E,\Z)$ which evaluate to 1 on each fiber. Such an element exists if and only if the Euler class of $\Pi$ vanishes. By Poincare duality, such elements are in one to one correspondence with homotopy classes  of sections of $\Pi$. Given such a class $\phi \in H^1(E,\Z)$, we may construct $\widehat E$ as the corresponding Abelian cover of $E$. Let $\beta:\widehat E \to E$ be the corresponding covering map and let $\widehat \Pi = \Pi \circ b$. Since $\phi$ evaluates to 1 on each fiber of $\Pi$, the $\widehat \Pi$ preimage of a point in $X$ is a copy of $\R$ as desired. Observe that under this construction, any section of $E$ Poincare dual to $\phi$ lifts to a section of $\widehat E$. In the context of this article, we are using the fiberwise lift corresponding to the $\infty$-section.

			Suppose now that $\Pi$ is equipped with a flat connection $\H$. Then on any fiberwise cover $\widehat \Pi$, there is an induced flat connection $\widehat{\H}$. One may define $\widehat \H$ as follows. A flat connection on the bundle $\Pi$ is the same thing as a foliation of $E$ by leaves transverse to the $S^1$ fibers; the flat connection may be recovered as holonomy of this foliation. The lift of this foliation to $\widehat E$ is transverse to the fibers of $\widehat \Pi$. Moreover, the holonomy of the lifted foliation exists along any path $\gamma$ in the base. This holonomy map is what we call $\widehat \H_\gamma$.
		\end{rmk}
		
		\begin{proof}[Proof of \cref{thm:main}]
			The flat connection $\widehat \J$ gives a homomorphism from $\pi_1(\MM)$ into $\Homeo^+(\R)$. This map is nontrivial since the monodromy around any closed orbit of the restriction of the suspension flow of $\phi$ to $\Mpunc$ is nontrivial\plabel{fix2:50}; indeed, the monodromy of $\Jpar$ around such a curve is a dilation. Theorem 1.1 of \cite{boyer_orderable_2002} states that for fundamental groups of irreducible 3-manifolds, the existence of any nontrivial map to a left-orderable group is equivalent to the existence of a left-ordering.
		\end{proof}
	\end{subsection}
	
	\begin{subsection}{The fiber of \titletex{$\widehat\Pi$} and the leaf space}
		\newcommand{\cover}[1]{\overline{#1}}
		In this section we will prove \cref{thm:main2}. Our strategy is to construct a monotone map from the leaf space of $\widetilde \F$ to a chosen fiber of $\widehat\Pi$. Before proceeding with the proof, we briefly summarize the relevant constructions from the previous section. $\Pi$ is an $S^1$ bundle over $\MM$ with trivial Euler class. $\J$ is a flat connection on $\Pi$. We defined a section $r_0$ of $\Pi$ called the zero-section. The zero-section is defined by the property that $\theta(r_0(p))=p$. The $\R$-bundle $\widehat \Pi$ is a fiberwise cover of $\Pi$, and is equipped with a flat connection $\widehat \J$. 

		\begin{proof}[Proof of \cref{thm:main2}]
			Let $\cover\Mpunc$ be the lift of $\Mpunc$ to $\widetilde\MM$.\plabel{fix2:52} We use bars to denote the lifts of objects in $\Mpunc$ (eg $\F^s$) to objects in $\cover \Mpunc$.  (eg $\cover {\F^s}$). Recall that $L$ is the leaf space of $\widetilde \F$.\plabel{fix3:30} Let $K$ be the the leaf space of $\cover {\F^s}$, where prongs incident with the same singularity are considered the same leaf. This is the same as saying that $K$ is the leaf space of the induced stable foliation in $\widetilde \MM$.\plabel{fix3:29} Let $f_1\colon L \to K$ be the monotone, $\pi_1(\MM)$-equivariant map which crushes each interval in $L$ corresponding with a lift of a saddle region down to a point.
		
			By construction, the $\infty$-section of the $S^1$-bundle $\Pi$ lifts to a section (not just a multi-valued section) of the $\R$-bundle $\widehat \Pi$. Refer to \cref{rmk:fiberwiselifts} for more discussion of this point. The zero-section is homotopic to the $\infty$ section, so it also lifts to a section of $\widehat \Pi$. \xlabel{fix7:2}{Choose such a lift and call it $\widehat{r_0}$.}\plabel{fix2:53} 

		Select a basepoint $p\in \Mpunc$ and a lift $\cover p \in \cover \Mpunc$. We can define a monotone, $\pi_1(\MM)$-equivariant map
		\[
		f_2 \colon K\rightarrow{\widehat \Pi}^{-1}(p)
		\]
		as follows. For any leaf $\cover \lambda$ in $\cover {\F^s}$, choose a point $\cover q$ on $\cover \lambda$ and a path $\cover \gamma\colon [0,1]\to \cover \Mpunc$ from $\cover q$ to $\cover p$\plabel{fix3:34}. Let $q$, $\lambda$, and $\gamma$ be the projections of $\cover q$, $\cover \lambda$, and $\cover \gamma$ to $\Mpunc$\plabel{fix2:54}, and similarly let $\gamma$ be the projection of $\cover \gamma$ to $\Mpunc$. We now define $$f_2(\cover \lambda)={\widehat \J_\gamma}(\widehat{r_0}(q)).$$ 

		Since $\widehat \J$ extends to a flat bundle over $\MM$, this definition is independent of the choice of $\gamma$ in its relative homotopy class in $\MM$. As discussed in \cref{lem:zeroflat}\plabel{fix3:33}, the zero-section is flat on $\lambda$. Therefore, our definition is independent of the choice of $q$ on $\lambda$. Finally, it's easy to check that two prongs of $\cover{\F^s}$ corresponding with the same point in $K$ have the same value of $f_2$. \xlabel{fix7:3}{One way to see this is to connect two adjacent prongs by a short path $\gamma_p$, note that $h_{\gamma_p}(\gamma_p(0))=\gamma_p(1)$, and then argue as in \cref{lem:zeroflat} that $f_2$ has the same value on these two prongs. Then one could iterate to show that all prongs incident to the same singularity have the same value of $f_2$.}

		Let's check that $f_2$ is $\pi_1$-equivariant, ie that it intertwines the action on $K$ by deck transformations and the action on $\widehat \Pi^{-1}(p)$ induced by $\widehat \J$.\plabel{fix5:21} Given $g\in \pi_1(\MM)$, represent $g$ by a loop $\gamma_2$ in $\Mpunc$ based at $p$. \xlabel{fix7:4}{We remind the reader that $g$ acts on $\widehat{\Pi}^{-1}(p)$ as $\widehat \J_{\gamma_2^{-1}}$, as explained in \cref{sec:connectionintro}.} Now we have
		\begin{align*}
			f_2(g\cover\lambda)&=\widehat{\J}_{\gamma * \gamma_2^{-1}}(\widehat{r_0}(q))&\text{since $\gamma * \gamma_2^{-1}$ lifts to a path from $g\cover \lambda$ to $\cover p$}\\
			&=\widehat{\J}_{\gamma_2^{-1}} \circ \widehat{\J}_{\gamma}(\widehat{r_0}(q))\\
			&=\widehat{\J}_{\gamma_2^{-1}} f_2(\cover \lambda)\\
			&=g f_2(\cover\lambda) &\text{since $g$ acts on ${\widehat \Pi}^{-1}(p)$ as $\widehat \J_{\gamma_2^{-1}}$}
		\end{align*}\plabel{fix2:55}

		Next, we will check that for any leaf $\lambda$ of $\cover {\F^s}$, there is an interval in $K$ around $\lambda$ which $f_2$ maps homeomorphically onto an interval in $\widehat{\Pi}^{-1}(p)$. (Though $K$ is not a 1-manifold, the notion of interval still makes sense: we mean the lift to $K$ of an interval in $\cover{\Mpunc}$ which is transverse to $\cover {\F^s}$.) \plabel{fix3:36} Roughly speaking, \xlabel{fix6:12.1}{the local leaf space of $\overline {\F^s}$ is homeomorphic to the local leaf space of ${\F^s}$,} and $\theta$ gives rise to a homeomorphism between the local leaf space of $\F^s$ and an interval in a fiber of $\widehat \Pi$. Finally, $\widehat \J$ gives a homeomorphism from any fiber of $\widehat \Pi$ to $\widehat \Pi^{-1}(p)$. The function $f_2$ can be expressed as a composition of these maps, and hence is locally a homeomorphism.\plabel{fix5:22}
		
		In more detail, choose a short curve $\cover{\gamma_3}\colon [0,1]\to \cover\Mpunc$ that is transverse to $\cover{\F^s}$, lies in the strong unstable foliation, and passes through the leaf $\lambda$. Then $\cover \gamma_3$ parameterizes an interval in $K$ \xlabel{fix6:12}{containing $\lambda$}. We will check that $f_2$ is a homeomorphism on this interval. Let $\cover q=\cover \gamma_3(1)$. Let $\cover \gamma$ be a path from $\cover q$ to $\cover p$. Let $\gamma$, $\gamma_3$, and $q$ be the projections of $\cover \gamma$, $\cover \gamma_3$, and $\cover q$ to $\Mpunc$.\plabel{fix3:35} Since $\gamma_3$ lies in the strong unstable foliation, $\gamma_3$ is the $\widehat \theta_q$-image of some curve $\gamma_4$ in ${\widehat\Pi}^{-1}(q)$. Here, $\widehat \theta_q:\widehat{\Pi}^{-1}(q) \to \Mpunc$ is the composition of $\theta$ with the covering map $\widehat\Pi^{-1}(q)\to \Pi^{-1}(q)$. A priori, there are many choices for $\gamma_4$ because the covering map $\widehat\Pi^{-1}(q)\to \Pi^{-1}(q)$ is $\Z$ to 1. We eliminate this ambiguity by asking that 
		\begin{equation}\label{eqn:normalize}
		\gamma_4(1)=\widehat{r_0}(\gamma_3(1)).
		\end{equation}

		\xlabel{fix6:12.2}{
		Using the definition of $\J$ on the type 1 curve $\gamma_3$, we have
		\begin{align*}
			\theta \circ \J_{\gamma_3|_{[t,1]}}(r_0(\gamma_3(t))) &= \theta \circ r_0(\gamma_3(t))\\
			&= \gamma_3(t)
		\end{align*}
		Therefore,
		\begin{align}
		\widehat{\theta}_q \circ \widehat{\J}_{\gamma_3|_{[t,1]}}(\widehat{r}_0(\gamma_3(t))) &= \gamma_3(t)\\
		\widehat{\J}_{\gamma_3|_{[t,1]}}(\widehat{r}_0(\gamma_3(t))) &= \gamma_4(t)\label{eqn:x5}
		\end{align}
		A priori, \cref{eqn:x5} holds only up to deck transformations of the covering $\widehat{\Pi}^{-1}(q) \to \Pi^{-1}(q)$. However, the normalization in \cref{eqn:normalize} is designed so that \cref{eqn:x5} holds exactly.}
		
		Now let us compute $f_2$ on the desired interval. For any $t\in [0,1]$,
		\begin{align}
			f_2(\overline {\gamma_3}(t))&=\widehat{\J}_{\gamma} \circ \widehat{\J}_{\gamma_3|_{[t,1]}}(\widehat{r_0}(\gamma_3(t)))\label{eqn:x1}\\
			&=\widehat{\J}_\gamma (\gamma_4(t))\label{eqn:x2}\\
			&=\widehat{\J}_\gamma (\widehat\theta_q^{-1}(\gamma_3(t))) \label{eqn:x3}.
		\end{align}
		In \cref{eqn:x1} we used that ${\gamma_3}|_{[t,1]}*\gamma$ lifts to a path from $\cover \gamma_3(t)$ to $\cover p$. In \cref{eqn:x2}, we applied \cref{eqn:x5}. In \cref{eqn:x3}, we have used the fact that $\widehat{\theta}_q$ restricts to a homeomorphism between the images of $\gamma_4$ and $\gamma_3$. \cref{eqn:x3} expresses $f_2$ locally as a composition of the homeomorphisms $\widehat{\J}$ and $\widehat \theta^{-1}$, so $f_2$ is a homeomorphism of $\cover\gamma_3$ onto its image as desired.\plabel{fix5:23} The chosen leaf $\lambda$ was arbitrary, so it follows that $f_2$ is continuous and monotone. We conclude that the composition $f=f_2\circ f_1$ is a continuous, monotone, $\pi_1(\MM)$-equivariant map. The $\pi_1(\MM)$ action on the image is nontrivial as noted in the proof of \cref{thm:main}.
		\end{proof}
		\begin{rmk}
			The map $f_2$ can be visualized quite cleanly in the setting of \cref{fig:sec2}. Extend $\gamma$ linearly in both directions and choose $p=\gamma(0)$. The vertical red lines (plus their completions at $\infty$) are $\theta$-images of fibers of $\Pi$. Lift these fibers to half-open subintervals of fibers of the $\R$-bundle, $\widehat \Pi$. Choose the lifts so that they intersect the zero-section in $\widehat \Pi$. Now parallel transport these subintervals to $\widehat{\Pi}^{-1}(\gamma(0))$ along $\gamma^{-1}$. The transported intervals are disjoint and cover ${\widehat \Pi}^{-1}(\gamma(0))$. To be more precise, let $\dots$, $x_{-1}$, $x_0$, $x_1$, $x_2$, \dots be the lifts of the point $x \in \Pi^{-1}(\gamma(0))$ to $\widehat{\Pi}^{-1}(\gamma(0))$. Parallel transport along $\gamma^{-1}$ maps the first vertical red line (viewed as a half-open interval in a fiber of $\widehat{\Pi}$) to $[x_{-1},x_0)$. The second vertical red line is transported to $[x_{-2},x_{-1})$, etc. During this parallel transport, the points at $\pm \infty$ of the vertical red lines follow the section $s_x$ drawn in the figure.\plabel{fix2:56}

			Therefore, the vertical lines (plus their completions at $\infty$) contain a representative from each point preimage of $f_2$. The quotient of the leaf space we have constructed (ie $\widehat{\Pi}^{-1}(\gamma(0))$) can now be seen as the concatenation of all of the vertical red lines in \cref{fig:sec2}. \xlabel{fix7:8}{This picture is analogous to that of the step map in the setting of skew-Anosov flows: the vertical red lines are unstable leaves in the orbit space, while the curved red arcs are analogous to stable leaves which make perfect fits with the unstable leaves. See \cite[Prop. 3.22]{barthelme_anosov_2017} for a discussion of the skew Anosov picture.}
		\end{rmk}
		
	\end{subsection}
\end{section}

\begin{section}{Computations}\label{sec:computations}
	Building on work of Dunfield and Bell, we were able to find \mfldcount{} manifolds in the Hodgson-Weeks census which can be constructed by a surgery satisfying the hypotheses of \cref{thm:main}. This represents about \mfldpercent{}\% of the 5801 non-L-spaces in the Hodgson-Weeks census \cite{hodgson_symmetries_1994, dunfield_floer_2019}. Dunfield and Bell found monodromies for many of the fibered, orientable 1-cusped manifolds that can be triangulated with at most 9 tetrahedra \cite{bell_data_nodate}. Using Bell's program \texttt{flipper}, they were able to find invariant laminations for about 25,700 of them \cite{bell_flipper_2013}. About 800 of these have orientable invariant laminations and monodromy preserving these orientations. The first few such examples with genus $\geq 2$ are listed in \cref{tab:cusped}.

	\begin{table}[htb]
		\centering
		\begin{tabular}{ll}
			\toprule
			Name & Genus \\
			\midrule
			m038 & 2  \\
			m120 & 3 \\
			s090 & 4 \\
			v0224 & 5 \\
			m221 & 3 \\
			t00448 & 6 \\
			o9\_00896 & 7 \\
			s173 & 4\\
			v0248 & 6\\
			m289 & 2\\
			t00682 & 4\\
			m305 & 2\\
			s296 & 2\\
			m310 & 3\\
			t00707 & 3\\
			\bottomrule
		\end{tabular}
		\caption{The first few 1-cusped fibered manifolds with genus $\geq 2$ and monodromy satisfying the conditions of \cref{thm:main}.}\label{tab:cusped}
	\end{table}

	We used \texttt{flipper} to drill these manifolds along their pseudo-Anosov singularities and produce ideal triangulations of the resulting many-cusped manifolds. Using \texttt{SnapPy}, we performed surgeries with small coefficients on these manifolds satisfying the constraints of \cref{thm:main} and identified the resulting manifolds in the Hodgson-Weeks census \cite{culler_snappy_nodate}. We found \mfldcount{} manifolds in the census, the first of which are shown in \cref{tab:closed}.

	\begin{table}[htb]
		\centering
		\begin{tabular}{lll}
			\toprule
			Name & Underlying fibered manifold & Volume\\
			\midrule
			m003(-2,3) & m004   & 0.981\\
			m004(6,1)  & m004   & 1.284\\
			m004(1,2)  & m004   & 1.398\\
			m003(-3,4) & v0650  & 1.414\\
			m009(4,1)  & m023   & 1.414\\
			m004(3,2)  & m004   & 1.440\\
			m004(7,1)  & m004   & 1.463\\
			m004(5,2)  & m004   & 1.529\\
			m015(5,1)  & t03310 & 1.757\\
			m009(5,1)  & m009   & 1.831\\
			m009(-5,1) & m009   & 1.831\\
			m011(1,3)  & v1577  & 1.831\\
			m009(1,2)  & m009   & 1.843\\
			m007(-5,1) & o9\_31045& 1.843\\
			m006(-5,1) & m009    & 1.941\\
			\bottomrule
		\end{tabular}
		\caption{The first few (closed) manifolds in the Hodgson-Weeks census to which \cref{thm:main} applies.}\label{tab:closed}
	\end{table}

	Dunfield obtained orderability results for many manifolds in the Hodgson-Weeks census either by constructing a taut foliation with vanishing Euler class or by constructing a $\PSL(2,\R)$ representation that lifts to a $\widetilde{\PSL}(2,\R)$ representation \cite{dunfield_floer_2019}. \cref{tab:overlap} shows the overlap in applicability between these methods and ours.

	\begin{table}[htb]
		\centering
		\begin{tabular}{ll|ll}
			&   & \multicolumn{2}{c}{\cref{thm:main} applies?}\\
		 	&   & Yes & No\\
			\midrule
			\multirow{2}{13em}{Taut foliation or $\PSL(2,\R)$ rep with Euler class 0?} & Yes  & 1795   & 1730\\
			&No  & 803   & 1473\\
			\bottomrule
		\end{tabular}
		\caption{The overlap in applicability between our method and previously used methods for proving orderability on the non-L-space rational homology spheres in the Hodgson-Weeks census.}\label{tab:overlap}
	\end{table}
\end{section}

\begin{section}{Remarks and questions}
	\begin{itemize}[wide]
		\item For which taut foliations does there exist a $\pi_1$-equivariant, order preserving map from the leaf space of the universal cover to $\R$? As a first step, we suggest the following conjecture:
			\begin{conj}\label{conj:anyslopes}
				\cref{thm:main2} holds without the condition that the surgery slopes have the same sign.
			\end{conj}
			This would greatly expand the scope of the results in this paper; for example, by \cite{fried_transitive_1983} it would include every 3-manifold carrying a transitive pseudo-Anosov flow with orientable invariant foliations. The difficulty is that the some of dilation factors $\lambda^{m_i p_i/q_i}$ may be greater than 1 in absolute value while others may be smaller than 1. Parallel sections are no longer monotonic as shown in \cref{fig:blowup}, and they might blow up as $t\to +\infty$ or as $t\to -\infty$. A more detailed study of their dynamics in this case is necessary.\plabel{fix2:57}

		\item Can the map $f\colon L\to \R$ be upgraded to a strictly monotone map? We expect that $f$ factors as 
			\begin{equation}
				\begin{tikzcd}
					L \arrow[r, "\mu"] \arrow[bend left, "f"]{rr} & R \arrow[r, "\eta"] \arrow[d, equal]&   \widehat{\Pi}\inv(p) \arrow[d,equal]\\
					& \R & \R
				\end{tikzcd}
			\end{equation}
			where $\mu$ is locally a homeomorphism onto its image, $\eta$ is monotonic, and $R$ is yet to be defined. The composition $f=\eta\circ \mu$ collapses the leaves of each saddle region in $L$ to a point. Therefore, $R$ should be obtained from $\widehat{\Pi} \inv (p)$ by blowing up the $f$-image of a leaf in a saddle region to a closed interval. The difficulty is that different saddle regions could conceivably map to the same point under $f$.

		\item What is the best possible analytic quality of the representations we have constructed?

		\item A generic pseudo-Anosov map will violate the orientability constraints of \cref{thm:main}. For surgery coefficients satisfying an appropriate parity condition, we expect that the methods of this paper may be extended to give an action of $\pi_1$ on $\R$ by possibly orientation-reversing homeomorphisms. 

		\item What can be said about fillings along the degeneracy slope? Experiments suggest that these are L-spaces, and hence do not carry taut foliations.
	\end{itemize}
\end{section}

\clearpage
\printbibliography

@book{kronheimer_monopoles_2007,
	title = {Monopoles and three-manifolds},
	volume = {10},
	publisher = {Cambridge University Press Cambridge},
	author = {Kronheimer, Peter B and Mrowka, Tomasz},
	year = {2007},
}

@article{roussarie_plongements_1974,
	title = {Plongements dans les variétés feuilletées et classification de feuilletages sans holonomie},
	volume = {43},
	url = {http://www.numdam.org/item/PMIHES_1974__43__101_0/},
	language = {en},
	urldate = {2020-04-23},
	journal = {Publications Mathématiques de l'IHÉS},
	author = {Roussarie, Robert},
	year = {1974},
	pages = {101--141},
}

@article{zorich_flat_2006,
	title = {Flat {Surfaces}},
	url = {http://arxiv.org/abs/math/0609392},
	urldate = {2020-04-21},
	journal = {arXiv:math/0609392},
	author = {Zorich, Anton},
	month = sep,
	year = {2006},
	note = {arXiv: math/0609392},
}

@article{hu_euler_2019,
	title = {Euler class of taut foliations and {Dehn} filling},
	url = {http://arxiv.org/abs/1912.01645},
	urldate = {2020-04-21},
	journal = {arXiv:1912.01645 [math]},
	author = {Hu, Ying},
	month = dec,
	year = {2019},
	note = {arXiv: 1912.01645},
}

@article{clay_graph_2013,
	title = {Graph manifolds, left-orderability and amalgamation},
	volume = {13},
	issn = {1472-2739, 1472-2747},
	url = {http://arxiv.org/abs/1106.0486},
	doi = {10.2140/agt.2013.13.2347},
	number = {4},
	urldate = {2020-04-18},
	journal = {Algebraic \& Geometric Topology},
	author = {Clay, Adam and Lidman, Tye and Watson, Liam},
	month = jul,
	year = {2013},
	note = {arXiv: 1106.0486},
	pages = {2347--2368},
}

@article{fried_transitive_1983,
	title = {Transitive {Anosov} flows and pseudo-{Anosov} maps},
	volume = {22},
	issn = {0040-9383},
	url = {http://www.sciencedirect.com/science/article/pii/0040938383900150},
	doi = {10.1016/0040-9383(83)90015-0},
	language = {en},
	number = {3},
	urldate = {2020-04-09},
	journal = {Topology},
	author = {Fried, David},
	month = jan,
	year = {1983},
	pages = {299--303},
}

@article{calegari_problems_2002,
	title = {Problems in foliations and laminations of 3-manifolds},
	url = {http://arxiv.org/abs/math/0209081},
	urldate = {2020-02-24},
	journal = {arXiv:math/0209081},
	author = {Calegari, Danny},
	month = sep,
	year = {2002},
	note = {arXiv: math/0209081},
}

@book{bell_flipper_2013,
	title = {flipper ({Computer} {Software})},
	author = {Bell, Mark},
	year = {2013},
	note = {Published: pypi.python.org/pypi/flipper},
}

@article{boyer_l-spaces_2011,
	title = {On {L}-spaces and left-orderable fundamental groups},
	url = {http://arxiv.org/abs/1107.5016},
	urldate = {2020-01-22},
	journal = {arXiv:1107.5016 [math]},
	author = {Boyer, Steven and Gordon, Cameron McA and Watson, Liam},
	month = jul,
	year = {2011},
	note = {arXiv: 1107.5016},
}

@article{ozsvath_holomorphic_2004,
	title = {Holomorphic disks and genus bounds},
	volume = {8},
	issn = {1364-0380},
	url = {https://msp.org/gt/2004/8-1/p08.xhtml},
	doi = {10.2140/gt.2004.8.311},
	number = {1},
	urldate = {2020-01-22},
	journal = {Geometry \& Topology},
	author = {Ozsvath, Peter and Szabo, Zoltan},
	month = feb,
	year = {2004},
	pages = {311--334},
}

@article{thurston_three-manifolds_1997,
	title = {Three-manifolds, {Foliations} and {Circles}, {I}},
	url = {http://arxiv.org/abs/math/9712268},
	urldate = {2020-01-22},
	journal = {arXiv:math/9712268},
	author = {Thurston, William P.},
	month = dec,
	year = {1997},
	note = {arXiv: math/9712268},
}

@article{boyer_orderable_2002,
	title = {Orderable 3-manifold groups},
	url = {http://arxiv.org/abs/math/0211110},
	urldate = {2019-09-26},
	journal = {arXiv:math/0211110},
	author = {Boyer, Steven and Rolfsen, Dale and Wiest, Bert},
	month = nov,
	year = {2002},
	note = {arXiv: math/0211110},
}

@book{penner_combinatorics_2016,
	title = {Combinatorics of {Train} {Tracks}. ({AM}-125)},
	isbn = {978-1-4008-8245-8},
	language = {en},
	publisher = {Princeton University Press},
	author = {Penner, R. C. and Harer, John L.},
	month = mar,
	year = {2016},
	note = {Google-Books-ID: L8XdCwAAQBAJ},
}

@article{culler_orderability_2018,
	title = {Orderability and {Dehn} filling},
	volume = {22},
	issn = {1364-0380, 1465-3060},
	url = {http://arxiv.org/abs/1602.03793},
	doi = {10.2140/gt.2018.22.1405},
	number = {3},
	urldate = {2021-09-25},
	journal = {Geometry \& Topology},
	author = {Culler, Marc and Dunfield, Nathan M.},
	month = mar,
	year = {2018},
	note = {arXiv: 1602.03793},
	pages = {1405--1457},
}

@book{calegari_foliations_2007,
	address = {Oxford},
	series = {Oxford mathematical monographs},
	title = {Foliations and the geometry of 3-manifolds},
	isbn = {978-0-19-857008-0},
	publisher = {Clarendon},
	author = {Calegari, Danny},
	year = {2007},
	note = {OCLC: ocm80331744},
}

@article{novikov_topology_1965,
	title = {The topology of foliations},
	volume = {14},
	journal = {Trudy Moskovskogo Matematicheskogo Obshchestva},
	author = {Novikov, Sergei Petrovich},
	year = {1965},
	pages = {248--278},
}

@article{boyer_taut_2019,
	title = {Taut foliations in branched cyclic covers and left-orderable groups},
	volume = {372},
	issn = {0002-9947, 1088-6850},
	url = {https://www.ams.org/tran/2019-372-11/S0002-9947-2019-07833-9/},
	doi = {10.1090/tran/7833},
	language = {en},
	number = {11},
	urldate = {2021-07-27},
	journal = {Transactions of the American Mathematical Society},
	author = {Boyer, Steven and Hu, Ying},
	year = {2019},
	pages = {7921--7957},
}

@article{calegari_laminations_2003,
	title = {Laminations and groups of homeomorphisms of the circle},
	doi = {10.1007/s00222-002-0271-6},
	author = {Calegari, Danny and Dunfield, N.},
	year = {2003},
}

@misc{bell_data_nodate,
	title = {Data accompanying the {Flipper} program, available at https://flipper.readthedocs.io/},
	author = {Bell, Mark C and Nathan, Dunfield},
}

@article{gabai_foliations_1983,
	title = {Foliations and the topology of 3-manifolds},
	volume = {18},
	issn = {0022-040X},
	url = {https://projecteuclid.org/euclid.jdg/1214437784},
	doi = {10.4310/jdg/1214437784},
	language = {EN},
	number = {3},
	urldate = {2020-06-10},
	journal = {Journal of Differential Geometry},
	author = {Gabai, David},
	year = {1983},
	mrnumber = {MR723813},
	zmnumber = {0533.57013},
	note = {Publisher: Lehigh University},
	pages = {445--503},
}

@article{fenley_structure_1998,
	title = {The structure of branching in {Anosov} flows of 3-manifolds},
	volume = {73},
	issn = {1420-8946},
	url = {https://doi.org/10.1007/s000140050055},
	doi = {10.1007/s000140050055},
	language = {en},
	number = {2},
	urldate = {2020-06-05},
	journal = {Commentarii Mathematici Helvetici},
	author = {Fenley, S. R.},
	month = jun,
	year = {1998},
	pages = {259--297},
}

@misc{culler_snappy_nodate,
	title = {{SnapPy}, a computer program for studying the geometry and topology of \$3\$-manifolds},
	author = {Culler, Marc and Dunfield, Nathan M. and Goerner, Matthias and Weeks, Jeffrey R.},
	note = {Published: Available at http://snappy.computop.org (01/07/2019)},
}

@article{dunfield_floer_2019,
	title = {Floer homology, group orderability, and taut foliations of hyperbolic 3-manifolds},
	url = {http://arxiv.org/abs/1904.04628},
	urldate = {2020-05-28},
	journal = {arXiv:1904.04628 [math]},
	author = {Dunfield, Nathan M.},
	month = nov,
	year = {2019},
	note = {arXiv: 1904.04628},
}

@article{hodgson_symmetries_1994,
	title = {Symmetries, isometries and length spectra of closed hyperbolic three-manifolds},
	volume = {3},
	issn = {1058-6458, 1944-950X},
	url = {https://projecteuclid.org/euclid.em/1048515809},
	language = {en},
	number = {4},
	urldate = {2020-05-28},
	journal = {Experimental Mathematics},
	author = {Hodgson, Craig D. and Weeks, Jeffrey R.},
	year = {1994},
	mrnumber = {MR1341719},
	zmnumber = {0841.57020},
	note = {Publisher: A K Peters, Ltd.},
	pages = {261--274},
}

@article{eisenbud_transverse_1981,
	title = {Transverse foliations of {Seifert} bundles and self homeomorphism of the circle},
	volume = {56},
	issn = {1420-8946},
	url = {https://doi.org/10.1007/BF02566232},
	doi = {10.1007/BF02566232},
	language = {en},
	number = {1},
	urldate = {2020-04-23},
	journal = {Commentarii Mathematici Helvetici},
	author = {Eisenbud, David and Hirsch, Ulrich and Neumann, Walter},
	month = dec,
	year = {1981},
	pages = {638--660},
}

@article{thurston_norm_1986,
	title = {A norm for the homology of 3-manifolds},
	volume = {59},
	issn = {0065-9266},
	number = {339},
	journal = {A norm for the homology of 3-manifolds},
	author = {Thurston, William P.},
	year = {1986},
	note = {Place: Providence, RI
Publisher: American Mathematical Society},
	pages = {99--130},
}

@article{kronheimer_scalar_1997,
	title = {Scalar curvature and the {Thurston} norm},
	volume = {4},
	doi = {10.4310/MRL.1997.v4.n6.a12},
	journal = {Mathematical Research Letters},
	author = {Kronheimer, P. and Mrowka, Tomasz},
	month = jan,
	year = {1997},
}

@article{barthelme_anosov_2017,
	title = {Anosov flows in dimension 3 - {Preliminary} version},
	url = {http://www.crm.umontreal.ca/sms/2017/pdf/diapos/Anosov_flows_in_3_manifolds.pdf},
	language = {en},
	author = {Barthelme, Thomas},
	year = {2017},
}
\end{document}